\documentclass[letterpaper, reqno,11pt]{article}
\usepackage[margin=1.0in]{geometry}
\usepackage{lineno}

\usepackage{blindtext, subfiles}
\usepackage[shortlabels]{enumitem}

\usepackage{tikz}
\usetikzlibrary{patterns, arrows.meta}

\usepackage{packages/commonpackage}
\usepackage{packages/package_for_list_of_content}
\usepackage{packages/package_to_delete}
\usepackage{packages/macro}

\newcommand{\RR}{\mathbb{R}}
\newcommand{\CC}{\mathbb{C}}

\newcommand{\eps}{\varepsilon}
\newcommand{\vardbtilde}[1]{\tilde{\raisebox{0pt}[0.85\height]{$\tilde{#1}$}}}
\newcommand{\erdogan}{Erdo\smash{\u{g}}an }

\numberwithin{equation}{section}
\theoremstyle{plain}
\newtheorem{theorem}{Theorem}[section]
\newtheorem{proposition}[theorem]{Proposition}
\newtheorem{corollary}[theorem]{Corollary}
\newtheorem{lemma}[theorem]{Lemma}

\theoremstyle{remark}
\newtheorem{definition}{Definition}[section]
\newtheorem{remark}{Remark}[section]
\newtheorem{example}{Example}[section]

\newtheorem{conjecture}[theorem]{Conjecture}
\newtheorem{problem}{Problem}[section]

\title{On planar Nikodym-type sets}
\author{
\begingroup
\setlength{\tabcolsep}{10pt} 
\begin{tabular}{ccc}
Alan Chang & Mingfeng Chen & Shaoming Guo  \\ Minxing Shen & Tongou Yang & Joshua Zahl
\end{tabular}
\endgroup
}

\date{}

\begin{document}
\maketitle

\begin{abstract}
    Nikodym \cite{Nik27} constructed a measure-zero Borel set in the plane that contains a punctured line segment passing through each point in the unit square. Fenton \cite{Lep76} and Falconer \cite{Fal85} asked whether a similar construction is possible when ``lines through a point'' are replaced by ``circles or circular arcs with a prescribed center.'' This was answered in the negative by Bourgain \cite{Bou86} and Marstrand \cite{Mar87}. We study generalizations of this question, in which ``lines through a point'' or ``circles with prescribed center'' are replaced by more general analytic families of plane curves.

    As a key ingredient, we prove new maximal function and local smoothing estimates for families of curves satisfying a multi-parameter analogue of Sogge's cinematic curvature condition \cite{Sog91}; this allows us to establish the non-existence of Nikodym-type sets in certain specific settings. The result also resolves a local smoothing conjecture in \cite{Zah26}. In the other direction, we use a recent nonlinear generalization of Falconer's prescribed projection theorem \cite{Fal86} to construct Nikodym-type sets in other specific settings. Combining these ingredients with tools from tame geometry, we precisely characterize when Nikodym-type sets exist for analytic families of plane curves. 
\end{abstract}


\section{Introduction}

\subsection{Statement of the problem}
In \cite{Nik27}, Nikodym constructed a Borel set $E\subset[0,1]^2$ with the following remarkable property: $E$ has zero two-dimensional Lebesgue measure, and for each $z\in[0,1]^2\backslash E$, there is a line $\ell$ containing $z$ so that 
\begin{equation}
\ell\cap E = \ell\cap[0,1]^2\backslash\{z\}.
\end{equation}
 In particular, if $z\in[0,1]^2\backslash E$ is not a corner of $[0,1]^2$ then $\mathcal{H}^1(\ell\cap E)>0$. Such sets are now called Nikodym sets.\footnote{Sometimes the set $[0,1]^2\backslash E$ is referred to as a Nikodym set.} 

Fenton \cite{Lep76} and independently Falconer \cite[Chapter 7]{Fal85} asked what happens when lines are replaced by circles: does there exist a Borel set $E\subset\RR^2$ with $|E|=0$ that contains a circle with every center? Or more modestly, does there exist a Borel set with $|E|=0$ that contains a set of circular arcs whose centers form a set of positive measure? Bourgain \cite{Bou86} and Marstrand \cite{Mar87} independently proved that no such set exists.

In this paper, we consider a generalization of Fenton and Falconer's question in which lines or circles are replaced by other families of curves. The tools to answer this question have recently become available, due to the following advances in harmonic analysis and geometric measure theory:
\begin{itemize}
	\item The Bourgain-Demeter-Guth $\ell^2$ decoupling theorems \cite{BDG16}, which we use to prove new local smoothing estimates. These are used to establish the nonexistence of Nikodym-type sets.

	\item Recent nonlinear generalizations \cite{BT23, CC19, CMT23, CFMMT} of Davies' efficient covering theorem \cite{Dav52} and Falconer's work \cite{Fal86} on prescribed projections. These are used to establish the existence of Nikodym-type sets.
\end{itemize}

Though it is possible to state our results using the notation and terminology of (smooth) families of smooth curves, it will be helpful to explicitly parameterize our curves as the graphs of smooth functions above a common closed interval. Our setup is as follows. Here and throughout, all intervals are assumed to have positive, finite length unless explicitly stated otherwise.
\begin{definition}\label{familyOfPlaneCurves}
Let $I\subset\RR$ be a closed interval and let $W\subset\RR^n\times\RR^d$. We will use coordinates $(z,u)\in\RR^n\times\RR^d$ to denote points in $W$. Let $\gamma\colon I\times W \to\RR$ be smooth. For each $(z,u)\in W$, define the curve
\[
\beta_{z,u} = \{(t, \gamma(t, z,u))\colon t\in I\}.
\]
If the defining function is ambiguous, we might write this as $\beta^\gamma_{z,u}$. For $z\in\RR^n$, define $\Gamma_z = \{\beta_{z, u}\colon (z,u)\in W\}$, and define $\Gamma = \{\Gamma_z\colon z\in\RR^n\}.$ We call $\Gamma$ a \emph{family of plane curves} associated to the defining function $\gamma$.
\end{definition}
For example, if $n=2$ and we write $z=(z_1,z_2)$ and
\begin{equation}
\gamma_{\textrm{lines}}(t, z, u)=z_2 + u(t-z_1),
\end{equation}
then $\beta_{z,u}$ is the line passing through $z$ with slope $u$ (restricted to the interval $I$). As a second example, if $I = [-1/10, 1/10]$, $W = B(0, 1/10)\times(1,2),$ and 
\begin{equation}
\gamma_{\textrm{circles}}(t, z, u)=z_2 + \sqrt{u^2 - (t-z_1)^2},
\end{equation}
then $\beta_{z,u}$ is the upper arc of the circle with center $z$ and radius $u$, restricted to the interval $[-1/10,1/10]$.

\begin{definition}\label{definitionOfNikodymSet}
Let $\Gamma$ be a family of plane curves associated to the defining function $\gamma$, which has domain $I\times W$. We say that $\Gamma$ \emph{admits Nikodym sets} if there are Borel sets $E\subset\RR^2$ and $F\subset\RR^n $ with $|E| = 0$ and $|F|>0$, so that for each $z\in F$, there exists $u$ with $(z,u)\in W$ so that $\mathcal{H}^1(\beta_{z,u}\cap E)>0$. 
\end{definition}
Nikodym's construction shows that $\Gamma_{\textrm{lines}}$ admits Nikodym sets, while Bourgain \cite{Bou86} and Marstrand \cite{Mar87} proved that $\Gamma_{\textrm{circles}}$ does not admit Nikodym sets.

In this paper, we investigate the following question: which families of plane curves admit Nikodym sets? When $n=1$, a construction by Wisewell\footnote{Wisewell in fact proved a more general result, see Theorem \ref{WisewellThm}.} shows that we always have Nikodym sets, unless the defining function $\gamma$ does not depend on $u$. 

We will mostly be concerned with the case $n=2$. In Theorem \ref{characterizationOfNikodymSets}, we give a precise characterization of which families of curves admit Nikodym sets, under the additional assumption that $W$ is open and $\gamma$ is analytic in a neighbourhood of $I \times \overline W$. The assumption that $W$ is open is mostly for convenience. The assumption that $\gamma$ is analytic (or a similar regularity assumption) is likely necessary in order to obtain a satisfactory answer, since otherwise (say if we only assume that $\gamma$ is smooth) different parts of the domain of $\gamma$ might have dramatically different behavior.

Before stating Theorem \ref{characterizationOfNikodymSets}, we will discuss several key challenges and features of this problem, which motivate the precise statement.

\subsection{Maximal functions and cinematic curvature}\label{maximalFnAndCinematicSection}
Bourgain's proof that $\Gamma_{\textrm{circles}}$ does not admit Nikodym sets is a consequence of a stronger result about the Bourgain circular maximal function. For $f\colon\RR^2\to\CC$ and $z\in\RR^2$, we define the (truncated) Bourgain circular maximal function
\[
Bf(z) = \sup_{u\in[1,2]}\int_{C(z,u)}|f(w)|dw.
\]
In the above, $C(z,u)$ is the circle with center $z$ and radius $u$. Bourgain \cite{Bou86} proved that
\begin{equation}\label{260618e_geometric}
\Vert Bf\Vert_p \leq C_p\Vert f\Vert_p
\end{equation}
 for all $p>2$. As a consequence, $\Gamma_{\textrm{circles}}$ does not admit Nikodym sets. Bourgain's original proof is very intricate; later Mockenhaupt, Seeger and Sogge \cite{MSS92} gave a short proof of Bourgain's result using local smoothing estimates, and Schlag \cite{Sch98} gave a purely geometric proof of Bourgain's result. 

We will define the analogue of Bourgain's circular maximal function for a general family of curves associated to a defining function $\gamma$.
\begin{definition}\label{averagingAndMaximalOperators}
Given $\gamma\colon I\times W\to\RR$ as in Definition \ref{familyOfPlaneCurves}, $f\colon \RR^2 \to\CC$ smooth, and $(z,u)\in W$, define the averaging operator
\[
A^\gamma f(z, u) = \int_I f(t, \gamma(t, z, u))dt,
\]
and define the maximal operator
\begin{equation}\label{defnMGamma}
M^\gamma f(z) = \sup_u \left|\int_I f(t, \gamma(t, z, u))dt\right|,
\end{equation}
where the supremum is taken over all $u\in\RR^d$ with $(z,u)\in W$ (with this definition, the domain of $M^\gamma$ is the projection of $W\subset\RR^n\times\RR^d$ to the first $n$ coordinates; we can extend the domain to all of $\RR^n$ by defining $M^\gamma$ to be 0 elsewhere).
\end{definition}
Notice that $M^\gamma$ has the same $L^p$ mapping properties as the maximal operator
\[
\sup_u \int_I |f(t, \gamma(t, z, u))|dt.
\]
If $M^\gamma$ is bounded on $L^p(\RR^2)$ for some $p<\infty$, then the corresponding family of curves $\Gamma$ does not admit Nikodym sets. We can verify this as follows. Let $E\subset\RR^2$ be a null set. For $s>0$, let $F_s$ be the set of points $z$ for which there exists a curve $\beta_{z,u}\in\Gamma_z$ with $\mathcal{H}^1(\beta_{z,u}\cap E)\geq s$. It suffices to show that $|F_s|=0$. Let $E^\circ\supset E$ be an open set with $|E^\circ|<\eps$, and let $f_j\colon \RR^2\to\RR$ be a sequence of smooth functions with $0\leq f_j\leq 1_{E^\circ}$, and $f_j\nearrow 1_{E^\circ}$ pointwise. By monotone convergence, we have $\lim_{j\to\infty} M^\gamma f_j(z)\gtrsim s$ for $z\in F_s$ (since the maximal operator $M^\gamma$ is defined with respect to integration in $t$, the implicit constant depends on $\sup |\partial_t  \gamma(t, z, u)|$), and hence 
 \begin{equation}
\Vert M^\gamma\Vert_{L^p\to L^p}\geq \sup_{j}\frac{\Vert M^\gamma f_j\Vert_p}{\Vert f_j\Vert_p} \gtrsim \eps^{-1/p}s|F_s|^{1/p}.
 \end{equation}
  Since $\eps>0$ was arbitrary, we conclude that $|F_s|=0$.\\

In \cite{Sog91, MSS92, MSS93}, Sogge and Mockenhaupt-Seeger-Sogge generalized Bourgain's result on the Bourgain circular maximal function to a broader class of curves that obey a property called cinematic curvature. In brief, a pair of distinct circles with centers $z,z'\in\RR^2$ and radii $u,u'\in[1,2]$ can be tangent, but if so, then their curvature must differ at the point of tangency. Indeed, if we (locally) represent the circles as the graphs of functions $\gamma(t, z, u)$ and $\gamma(t, z', u')$, then 
\begin{multline}\label{SoggeForbid}
\inf_t \big(|\gamma(t, z, u)-\gamma(t, z', u')| + |\partial_t\gamma(t, z, u) -\partial_t\gamma(t, z', u')| \\
+ |\partial_t^2\gamma(t, z, u) -\partial_t^2\gamma(t, z', u')|\big)
\gtrsim \big|(z,u) - (z', u')\big|.
\end{multline}
Writing $z=(z_1,z_2)$ and localizing to a suitably small region, this is the same as requiring that for all $(t, z, u)\in I\times W$, we have
\begin{equation}\label{soggeLocal}
\det \Bigg(\begin{array}{ccc}
\partial_{z_1}\gamma & \partial_{z_2}\gamma & \partial_u\gamma\\
\partial_{z_1}\partial_t\gamma & \partial_{z_2}\partial_t\gamma & \partial_u\partial_t\gamma\\
\partial_{z_1}\partial_t^2\gamma & \partial_{z_2}\partial_t^2\gamma & \partial_u\partial_t^2\gamma\\
\end{array}\Bigg) \neq 0.
\end{equation}
 Families of curves $\Gamma$ whose defining function $\gamma$ satisfies \eqref{soggeLocal} are said to satisfy Sogge's cinematic curvature condition. Sogge and Mockenhaupt-Seeger-Sogge obtained local smoothing estimates for the associated averaging operator $A^\gamma$, which in turn implied that the associated maximal function is bounded on $L^p(\RR^2)$ for every $p>2$. 

In the above discussion, it is important that circles are defined by three parameters. More generally, if $W\subset\RR^2\times\RR^d$ then we have a $d+2$ parameter family of curves, and the analogous multiparameter cinematic curvature condition is 
\begin{equation}\label{soggeLocalMultiParam}
\det \left(\begin{array}{ccccc}
\partial_{z_1}\gamma & \partial_{z_2} \gamma & \partial_{u_1} \gamma& \ldots& \partial_{u_d} \gamma\\
\partial_{z_1}\partial_t \gamma & \partial_{z_2}\partial_t \gamma & \partial_{u_1}\partial_t \gamma& \ldots&  \partial_{u_d}\partial_t \gamma\\
\partial_{z_1}\partial_t^2 \gamma & \partial_{z_2}\partial_t^2 \gamma & \partial_{u_1}\partial_t^2 \gamma& \ldots&  \partial_{u_d}\partial_t^2 \gamma\\
\vdots & \vdots & \vdots & \ddots & \vdots \\
\partial_{z_1}\partial_t^{d+1} \gamma & \partial_{z_2}\partial_t^{d+1} \gamma & \partial_{u_1}\partial_t^{d+1} \gamma& \ldots&  \partial_{u_d}\partial_t^{d+1} \gamma
\end{array}\right) \neq 0.
\end{equation}
Using tools from decoupling, we prove the following result, which may be of independent interest. Here and in what follows, if $W'\subset W$, then we write $\gamma|_{W'}$ to denote the restriction of $\gamma$ to $I\times W'$. 

 \begin{theorem}\label{maximalFnBound}
 Let $I\subset\RR$ be a closed interval, let $W\subset\RR^2\times\RR^d$ be open, let $\gamma\colon I\times W \to\RR$ be smooth. Suppose that $\gamma$ satisfies the multiparameter cinematic curvature condition \eqref{soggeLocalMultiParam} on $I\times W$. 

Let $W'\subset W$ be compact, and let $M^\gamma$ be the maximal operator associated to $\gamma|_{W'}$, in the sense of Definition \ref{averagingAndMaximalOperators}. Then $M^\gamma$ is bounded on $L^p(\RR^2)$ for all $p>2$ when $d=1$, and $p>d+2$ for $d>1$.
 \end{theorem}
 \begin{corollary}\label{corOfBddMaximalFn}
  Let $I\subset\RR$ be a closed interval, let $W\subset\RR^2\times\RR^d$ be open, let $\gamma\colon I\times W \to\RR$ be smooth. Suppose that $\gamma$ satisfies the multiparameter cinematic curvature condition \eqref{soggeLocalMultiParam} on $I\times W$. Then the family of curves associated to $\gamma$ does not admit Nikodym sets.
 \end{corollary}
  Corollary \ref{corOfBddMaximalFn} follows from Theorem \ref{maximalFnBound} when $d\geq 1$. For $d=0$, the result follows from a change of variables and Fubini. 

When $d=1$, Theorem \ref{maximalFnBound} was proved by Mockenhaupt-Seeger-Sogge \cite{MSS92}. A special case of Theorem \ref{maximalFnBound} was proved by Lee-Lee-Oh \cite{LLO25} in which $\Gamma$ is a family of ellipses, which corresponds to $d=3$; Lee-Lee-Oh established the result for $p>12$. Lee-Lee-Oh also considered the case where $\Gamma$ is a family of axis-parallel ellipses, which corresponds to $d=2$; Lee-Lee-Oh established the result for $p>4$. In Appendix \ref{ellipticMaximalOperatorSection}, we will combine Theorem \ref{maximalFnBound} with an estimate in \cite{PYZ22} and show that the axis-parallel elliptic maximal operator (first introduced by \erdogan \cite{Erd03}) is bounded on $L^p(\RR^2)$ in the sharp range $p>3$. 

Theorem \ref{maximalFnBound} will be a consequence of a stronger result, which says that the averaging operator $A^\gamma$ admits a nontrivial local smoothing estimate (see Theorem \ref{averagingOp} below), and the smoothly localized maximal operator $M^\gamma_\chi$ has a property called ``high frequency decay''. This resolves a local smoothing conjecture by the sixth author \cite{Zah26}. 

The range of Lebesgue exponents $p>d+2$ in Theorem \ref{maximalFnBound} is obtained by combining the high frequency decay estimate mentioned above (this estimate uses Fourier decoupling) with estimates from \cite{PYZ22} and \cite{Zah26}  obtained through purely geometric arguments.  For some special class of $\gamma$, this strategy can be employed more efficiently to prove that $M^\gamma$ is bounded in the range $p>d+1$, see Proposition \ref{230609theorem2_1} for an application of such a strategy to the axis-parallel elliptic maximal operator, and see \cite[Section 7]{Zah26} for further details. Moreover, examples in \cite{Zah26} show that for every $\gamma$ in Theorem \ref{maximalFnBound}, the operator $M^{\gamma}$ is unbounded on $L^p$ for every $p< d+1$. It is an interesting open problem to find the optimal range of $p$ for general $\gamma:$ 
\begin{problem}
What is the sharp range of $L^p$ bounds for the operator $M^\gamma$ described in Theorem \ref{maximalFnBound}?
\end{problem}

As discussed above, if we have an operator bound $\Vert M^\gamma f(z)\Vert_p\lesssim \Vert f\Vert_p$ for some $p<\infty,$ then the family of curves $\Gamma$ does \emph{not} admit Nikodym sets. Unfortunately the converse is false, as the next example illustrates:
\begin{example}\label{noNikodymUnboundedMaximal}
Let $I=[0,1]$, $W = (0,1)^2\times(0,1)$. Define 
\begin{equation}
\gamma(t,z,u) = z_2 + u(t-z_1)+u^2(t-z_1)^2.
\end{equation}
We will show in Appendix \ref{exampleNoNikodymUnboundedMaximalSection} that the family of curves $\Gamma$ associated to $\gamma$ does not admit Nikodym sets, but the maximal operator $M^\gamma$ is unbounded on $L^p(\RR^2)$ for all $p<\infty$. The reason for this is as follows. Let $r_n\searrow 0$ and let $K_n = [0,1]\times[0,r_n]$. We can construct a sequence of sets  $E_n\subset K_n$ and $F_n\subset K_n$ with $|E_n|/|K_n|\searrow 0$ and $|F_n|\geq |K_n|/100$ so that for each $z\in F_n$ there is a curve $\beta_{z,u}$ with $\mathcal{H}^1(\beta_{z,u}\cap E_n)\geq 1/100$. In particular, $M^\gamma(\chi_{E_n})\geq 1/400$ on $F_n$, and hence 
\begin{equation}
\lim_{n\to\infty} \Vert M^\gamma \chi_{E_n}\Vert_p / \Vert \chi_{E_n}\Vert_p= \infty.
\end{equation}

This example highlights an interesting subtlety in the study of generalized Nikodym sets. The sets $E_n$ are (rescaled) approximations to a classical Nikodym set---if $\tilde E_n$ denotes the image of $E_n$ under the rescaling sending $K_n$ to $[0,1]^2$, then $\tilde E_n$ converges to a classical Nikodym set. Thus the family of curves associated to $\gamma$ does not admit Nikodym sets, but it does admit a sequence of sets that become a Nikodym set in an appropriate rescaling limit. 
\end{example}

We conclude this section with a slight variant of Bourgain's circular maximal function, in which circles centered at $z$ are replaced by parabolas with vertex $z$.
\begin{example}\label{bourgainParabolasExample}
Let $I=[0,1]$, $W=(-1,1)^2\times(-1,1)$. Define $\gamma(t,z,u) = z_2 + u(t-z_1)^2$. For $(t, z, u)\in I\times W$, the determinant in \eqref{soggeLocal} evaluates to $4u$, so in particular the cinematic curvature vanishes on the surface $\{u=0\}$ within the domain of $\gamma$. However, it was shown in \cite{MR98} that the associated maximal operator $M^\gamma$ is bounded on $L^p(\RR^2)$ for all $p>2$. 
\end{example}
Examples \ref{noNikodymUnboundedMaximal} and \ref{bourgainParabolasExample} show that the (un)boundedness of $M^\gamma$ does not depend solely on the (non)vanishing of cinematic curvature on the domain (or closure of the domain) of $\gamma$. It is an interesting open problem to characterize those defining functions $\gamma$ for which the associated operator $M^\gamma$ is bounded on $L^p(\RR^2)$ for some $p<\infty$. In the one parameter, translation-invariant analytic case, i.e.~$d=1$ and $\gamma(t, z, u) = z_2 + h(z_1-t, u)$ for $h$ analytic, necessary and sufficient conditions for the boundedness of $M^\gamma$ were proved in \cite{CG24}. In this direction, we propose the following conjecture.
\begin{conjecture}
$M^\gamma$ is unbounded if and only if $\gamma$ admits Nikodym sets, possibly after taking an appropriate rescaling limit (in the spirit of Example \ref{noNikodymUnboundedMaximal}).
\end{conjecture}

\subsection{Local curve reuse}
Here is a key difference between Nikodym sets for lines (in the sense of Nikodym's original construction), and Nikodym sets for circles (whose existence is shown to be impossible): there are many pairs $(z,u)$ that give rise to the same line segment $\{(t, z_2+u(t-z_1))\colon t\in I\}$; indeed, two such pairs $(z,u)$ and $(z',u')$ give rise to the same line segment precisely when $u=u'$ and $z_2'-z_2=u(z_1'-z_1)$. On the other hand, each pair $(z,u)$ creates a distinct circle $C(z,u)$ (or circular arc). 

\begin{definition}\label{curveReuseDefn}
Let $\Gamma$ be a family of plane curves associated to the defining function $\gamma\colon I\times W\to\RR$. For each curve $\beta$, define
\begin{align*}
\hat\beta & = \{(z,u)\in W\colon \beta_{z,u}=\beta\},\\
\beta^* & = \pi_z(\hat\beta),
\end{align*}
where $\pi_z\colon \RR^2\times\RR^d\to\RR^2$ denotes the projection to the first two coordinates. If the domain $W$ is not clear from context, then we might write $\hat\beta^W$ for $\hat\beta\cap W$, and $(\beta^W)^*$ for $\pi_z(\hat\beta^W)$.

We say a curve $\beta$ is \emph{reused} (in $W$) if $\beta^*$ contains a set with positive one-dimensional Hausdorff measure, and we say $\beta$ is \emph{massively reused} if $\beta^*$ contains a set with positive two-dimensional Lebesgue measure\footnote{As we will see below, if $\gamma$ is analytic and $W$ is open, then $\beta$ reused implies that $\beta^*$ contains a smooth 1-dimensional manifold, while $\beta$ massively reused implies that $\beta^*$ contains an open set.}. We say that $\beta$ is \emph{locally reused near the point} $(z,u)$ if for every neighbourhood $W^\circ$ of $(z,u)$, we have that $\pi_z(\hat\beta\cap W^\circ)$ contains a set with positive one-dimensional Hausdorff measure. If the set $W$ is not clear from context, then we will say that $\beta$ is locally reused in $W$ near the point $(z,u)$.

We say that $\Gamma$ has \emph{curve reuse} if $\bigcup \beta^*$ contains a set with positive two-dimensional Lebesgue measure, where the union is taken over all reused curves.  We say that $\Gamma$ has \emph{massive curve reuse} if there is a curve that is massively reused. We say that $\Gamma$ has \emph{local curve reuse (in W)} if the set of points 
\[
\{z\colon\ \exists\ u\ \textrm{s.t.}\ \beta_{z,u}\ \textrm{has local curve reuse near}\ (z,u)\}
\]
contains a set with positive two-dimensional Lebesgue measure. As above, if the domain $W$ is not clear from context, then we will say $\Gamma^W$ has curve reuse (resp.~massive curve reuse, local curve reuse).
\end{definition}

\begin{remark}\label{curveReuseIsLocalRemark}
Whether a fixed curve is reused or massively reused, and whether $\Gamma$ admits Nikodym sets, are \emph{local properties}, in the following sense. Let $\gamma\colon I \times W\to\RR$, let $\{W_i\}$ be a countable cover of $W$ by open sets, and let $\Gamma_i$ be the set of curves corresponding to $\gamma|_{W_i}$. 
\begin{itemize}
\item A curve $\beta$ is reused (resp.~massively reused) in $W$ if and only if the same is true for at least one $W_i$.
\item $\Gamma$ admits Nikodym sets if and only if the same is true for at least one $\Gamma_i$. 
\end{itemize}
\end{remark}

If $\Gamma$ has massive curve reuse, then clearly $\Gamma$ admits Nikodym sets. The following example illustrates this phenomenon.
\begin{example}[Boring Nikodym sets]\label{dumbNikodymExample}
Let $I=[0,1]$, $W=[0,1]^3$. Let $\gamma(t, z, u) = z_2-u$, so $\beta_{z,u} = I\times\{z_2-u\}$. Let $E = \RR\times\{0\}$ and $F = [0,1]^2$. Then $|E| = 0$, $|F|>0$. For each $z\in F$, choose $u = z_2$. Then $\mathcal{H}^1(\beta_{z,u}\cap E)>0$. We conclude that the family of curves $\Gamma$ associated to $\gamma$ admits Nikodym sets.
\end{example}


\subsection{Curve reuse arising from foliations}
Next, we present two examples of families of curves that have local curve reuse, but do not admit Nikodym sets. 
\begin{example}[Boring non-Nikodym sets]\label{dumbNikodymNonExample}
Consider the family of horizontal lines $\gamma_{\operatorname{HL}}(t, z, u) = z_2$, so $\beta_{z,u}$ is the horizontal line segment $I\times\{z_2\}$. Every curve $\beta_{z,u}$ is reused, and $\gamma$ has local curve reuse. However, by Fubini, the family of curves $\Gamma$ associated to $\gamma$ does not admit Nikodym sets.
\end{example}

\begin{example}[Boring non-Nikodym sets]\label{dumbNikodymNonExample2}
Consider the family of unit parabolas $\gamma_{UP} (t, z, u) = (t-z_1)^2$, so $\beta_{z,u}$ is the unit parabola above the interval $I$ with vertex $(z_1,0)$. Every curve is reused, and $\gamma$ has local curve reuse. However the family of curves $\Gamma$ associated to $\gamma$ does not admit Nikodym sets.
\end{example}

We will introduce several definitions to help us better understand these examples and their implications. 
\begin{definition}
Let $\{V_i\}$ be a countable set of smooth vector fields from $\RR^2\to\RR^2$. We say a smooth curve $\beta$ is \emph{covered} by $\{V_i\}$ if for almost every $q\in\beta$, there is an index $i$ so that $V_i(q)\neq 0$ and $V_i(q)$ is parallel to the tangent space $T_q\beta$, i.e.,~if $\beta(s)$ is a unit-speed parameterization of $\beta$ along an interval $J$, then for a.e.~$s\in J$ there is an index $i$ so that $V_i(\beta(s))\neq 0$ and $V_i(\beta(s))\wedge \beta'(s)=0$.
\end{definition}

\begin{definition} \label{defnCompletelyCovered}
We say that a set of curves is \emph{completely covered by foliations} if there exists a countable set of vector fields $\{V_i\}$ so that every curve from the set is covered by $\{V_i\}$.
\end{definition}

\begin{definition}\label{coveredByFoliationsDefinition}
Let $\gamma$ be as above.
\begin{enumerate}[(I)]
\item We say that \emph{the reused curves in $\gamma$ are strongly covered by foliations} if there exists a countable set of vector fields $\{V_i\}$ so that every reused curve is covered by $\{V_i\}$.

\item We say that \emph{the reused curves in $\gamma$ are covered by foliations} if there exists a countable set of vector fields $\{V_i\}$ so that $\big|\bigcup \beta^*\big|=0$, where the union is taken over all reused curves that are not covered by $\{V_i\}$. 

\item We say that \emph{the locally reused curves in $\gamma$ are covered by foliations} if there exists a countable set of vector fields $\{V_i\}$ so that the set
\begin{equation}\label{setOfLocallyReusedCuvesCoveredByVi}
Q^\gamma(W,\{V_i\}) = \big\{z\colon \exists\ u\ \textrm{s.t.}\ \beta_{z,u}\ \textrm{has local curve reuse in $W$ near}\ (z,u),\ \textrm{and}\ \beta_{z,u}\ \textrm{is not covered by}\ \{V_i\}\big\}
\end{equation}
satisfies $|Q^\gamma(W,\{V_i\})|=0$.
\end{enumerate}
\end{definition}

\begin{definition}\label{curveReuseNotAECovered}
We say that $\gamma$ has \emph{local curve reuse not covered by foliations} if for every countable set of vector fields $\{V_i\}$, the set $Q^\gamma(W,\{V_i\})$ contains a set of positive two-dimensional Lebesgue measure. 
\end{definition}

The family of curves defined by $\gamma_{HL}$ from Example \ref{dumbNikodymNonExample} is covered by the constant vector field $V_1(z) = (1,0)$. The family of curves defined by $\gamma_{UP}$ from Example \ref{dumbNikodymNonExample2} cannot be covered by a single vector field. However, the curves from $\gamma_{UP}$ are covered by the pair of vector fields $\{V_1,V_2\}$, where $V_1(z) = \phi(z_2)(1, 2\sqrt{z_2})$ and $V_2(z)=\phi(z_2)(1,-2\sqrt{z_2})$, where $\phi(z_2) = e^{-1/z_2}$ for $z_2>0$ and $\phi(z_2)=0$ for $z_2\leq 0$.

\subsection{Which families of curves admit Nikodym sets?}

With the above definitions, we can now describe which families of curves admit Nikodym sets.

\begin{theorem}\label{NikodymIffCurveReuse}
Let $I\subset\RR$ be a bounded closed interval. Let $W\subset\RR^2\times\RR^d$ be open and let $\gamma\colon I\times W\to\RR$ be analytic on a neighbourhood of $I\times \bar W$. Let $\Gamma$ be a family of plane curves associated to the defining function $\gamma$. Then $\Gamma$ admits Nikodym sets if and only if at least one of the following holds: (i) $\Gamma$ has massive curve reuse, or (ii) $\Gamma$ has local curve reuse not covered by foliations.
\end{theorem}

Given a family of curves $\Gamma$, it may not be immediately obvious whether conditions (i) or (ii) are true. The next two results give a structural criterion for determining whether a family of curves admits Nikodym sets. Before stating these results, we will introduce several ancillary quantities.

\begin{definition}\label{DefnTauAndB}
Let $\gamma\colon I\times W\to\RR$, as in Definition \ref{familyOfPlaneCurves}. For $(z,u)\in W$ and $z=(z_1,z_2)$, define
\begin{equation*}
\begin{split}
\mathfrak{r}_c^\gamma(z, u) & = \operatorname{rank}\{\partial_{z_1}\gamma(\cdot, z, u),\ \partial_{z_2}\gamma(\cdot, z, u),\ \partial_{u_1}\gamma(\cdot, z, u),\ldots, \partial_{u_d}\gamma(\cdot, z, u)\},\\
\mathfrak{r}_p^\gamma(z, u) & = \operatorname{rank}\{\partial_{u_1}\gamma(\cdot, z, u),\ldots, \partial_{u_d}\gamma(\cdot, z, u)\}.
\end{split}
\end{equation*}
In the above, we think of $\partial_{z_1}\gamma(\cdot, z, u), \partial_{z_2}\gamma(\cdot, z, u)$, etc.~as functions of the variable $t$ (with $z$ and $u$ frozen), and $\operatorname{rank}\{f_1(t),\ldots,f_k(t)\}$ denotes the rank of this set of functions over $\RR$. If $d=0$ then $\mathfrak{r}_p^\gamma(z, u)$ is identically 0, and $\mathfrak{r}_c^\gamma(z, u) = \operatorname{rank}\{\partial_{z_1}\gamma(\cdot, z, u),\ \partial_{z_2}\gamma(\cdot, z, u)\}$. We will refer to the quantities $\mathfrak{r}_c^\gamma$ and $\mathfrak{r}_p^\gamma$ as the \emph{curve rank} and \emph{point rank}, respectively.
\end{definition}

\begin{theorem}\label{characterizationOfNikodymSets}
Let $I\subset\RR$ be a bounded closed interval. Let $W\subset\RR^2\times\RR^d$ be non-empty, open, and bounded. Let $\gamma\colon I\times W\to\RR$ be analytic. Let $\Gamma$ be the family of plane curves associated to the defining function $\gamma$. Suppose that $\mathfrak{r}_p^\gamma=m_p$ and $\mathfrak{r}_c^\gamma=m_c$ are constant on $W$. Then the behavior of $\Gamma$ is determined by the quantities $m_p$ and $m_c$ as follows.
\begin{enumerate}[(A)]
    \item\label{CaseMc=Mp+2} $m_c=m_p+2$:  $\Gamma$ does not admit Nikodym sets, and no curve from $\Gamma$ is reused.
    \item \label{CaseMc=mp+1}  $m_c = m_p+1$.
            \begin{enumerate}[(i)]
            \item\label{CaseMp=0} $m_p=0$: 
            $\Gamma$ does not admit Nikodym sets. $\Gamma$ does not have massive curve reuse. For every $(z,u)\in W$, the curve $\beta_{z,u}$ is reused near $(z,u)$. In particular, $\Gamma$ has local curve reuse.  The reused curves are strongly covered by foliations.

            \item\label{CaseMp>0} $m_p>0$:
            $\Gamma$ admits Nikodym sets. $\Gamma$ has local curve reuse, and the locally reused curves are not covered by foliations.
        \end{enumerate}
    \item\label{CaseMp=mc} $m_p=m_c$:
    $\Gamma$ admits Nikodym sets and has massive curve reuse.
\end{enumerate}
\end{theorem}

One of the hypotheses of Theorem \ref{characterizationOfNikodymSets} is that $\mathfrak{r}_p^\gamma$ and $\mathfrak{r}_c^\gamma$ are constant on $W$. The next result allows us to reduce to this situation.

\begin{proposition}\label{breakIntoConstantMpAndMc}
Let $I\subset\RR$ be a bounded closed interval. Let $W\subset\RR^2\times\RR^d$ be a bounded open ball (or more generally, a bounded set of the form $\{ (z,u)\colon f(z,u)<0\}$ where $f$ is analytic). Let $\gamma\colon I\times W\to\RR$ be analytic on a neighbourhood of $I\times\overline W$. Let $\Gamma$ be the family of plane curves associated to the defining function $\gamma$.

Then there is a finite collection of bounded, connected open sets $\tilde W_1,\ldots,\tilde W_N$, with each $\tilde W_i\subset \RR^2\times\RR^{d_i}$ (here $0\leq d_i\leq d$ is an integer) and analytic functions $\tilde \gamma_i\colon I\times \tilde W_i\to\RR$. These objects have the following properties:
\begin{enumerate}[(a)]
\item\label{NikodymIffOnGammaI} $\Gamma$ admits Nikodym sets if and only if the same is true for at least one $\tilde \Gamma_i$.
\item\label{constMcMpOnGammaI} $\mathfrak{r}_p^{\tilde \gamma_i}$ and $\mathfrak{r}_c^{\tilde \gamma_i}$ are constant on $\tilde W_i$.
\item\label{massiveCurveReuseIffOnGammaI} $\Gamma$ has massive curve reuse if and only if the same is true for at least one $\tilde \Gamma_i$.
\item\label{localReuseIffOnGammaI} Suppose that $\Gamma$ does not have massive curve reuse. Then $\Gamma$ has local curve reuse not covered by foliations if and only if the same is true for at least one $\tilde\Gamma_i$.
\end{enumerate}
\end{proposition}

Theorem \ref{NikodymIffCurveReuse} will follow from combining Theorem \ref{characterizationOfNikodymSets} and Proposition \ref{breakIntoConstantMpAndMc}.

\subsection{Structure of the paper}
In Section \ref{maximalFnSection} we will prove Theorem \ref{maximalFnBound}. This is one of the main new contributions in this paper\footnote{A version of this proof first appeared in the unpublished manuscript \cite{CGY23}, written by a subset of the authors.}. Theorem \ref{maximalFnBound} follows from a ``high frequency decay'' estimate, which in turn follows from a local smoothing estimate for the averaging operator $A^\gamma$. This is proved using decoupling.

In Section \ref{n=1Case} we will use a theorem of Wisewell to establish the existence of Nikodym sets in the case $n=1$. In Section \ref{curveReuseDaviesSection} we will introduce a nonlinear version of Falconer's prescribed projection theorem that was developed by Fraser, McDonald, Meyer, Taylor, and the first author. 

In Section \ref{semiAnalyticGeomSection} we will introduce tools from tame geometry that will allow us to break analytic and sub-analytic sets into well-behaved pieces. We prove Proposition \ref{breakIntoConstantMpAndMc} in this section. In Section \ref{curveRankSection} we will show how to reduce Theorem \ref{characterizationOfNikodymSets} to the special case where $\mathfrak{r}_p^\gamma=d$, and in Section  \ref{ProofOfThmCharacterizationOfNikodymSetsSection} we will prove Theorem \ref{characterizationOfNikodymSets}.

\subsection{Acknowledgements}
Alan Chang's research was partially supported by U.S.\ National Science Foundation grant DMS-2247233. Shaoming Guo's research is partly supported by NSF-2044828, and by the Nankai Zhide Foundation, NSFC Grant No. 12426204, the New Cornerstone Science Foundation, and the Fundamental Research Funds for the Central Universities No. 100-63263092. Minxing Shen's research was supported by the National Natural Science Foundation of China, grant number 12501125. Tongou Yang's research was supported in part by the National Key R\&D Program of China (No. 2024YFA1015400). Joshua Zahl's research was supported by Discovery and Alliance grants from the Natural Sciences and Engineering Research Council of Canada; by the Fundamental Research Funds for the
Central Universities (Project No. 100-63253272); and by the Nankai Zhide Foundation. This project was started when Shaoming Guo visited Malabika Pramanik and Joshua Zahl at UBC in 2022; he would like to thank them for the invitation and hospitality.

\subsection{AI usage}
All ideas and arguments in this manuscript are due to the authors. AI (specifically ChatGPT Astra) was used in the editing process to proof-read the arguments. In several instances AI identified a mathematical error in the arguments (in addition to numerous typos), which the authors corrected. All text was written by the authors, with the exception of a few specific wording suggestions proposed by AI, plus routine AI-assisted find/replace operations.


\section{Maximal functions, local smoothing, and multi-parameter cinematic curvature}\label{maximalFnSection}

Our goal in this section is to prove Theorem \ref{maximalFnBound}. The main technical step is to prove a local smoothing estimate for a smoothly cutoff version of the operator $A^\gamma$ from Definition \ref{averagingAndMaximalOperators}. In the definition below, we consider averaging operators whose domain is (ostensibly) all of $\RR\times\RR^2\times\RR^d$, but we restrict the domain using a smooth cutoff function.
\begin{definition}\label{mollifiedAGamma}
Let $\gamma\colon\RR\times\RR^2\times\RR^d\to\RR$ be smooth, and let $\chi\colon \RR\times\RR^2\times\RR^d\to\RR$ be a smooth bump function supported near the origin. Let $f\colon\RR^2\to\CC$. For $z\in\RR^2,\ u\in\RR^d$, define the averaging operator
\[
A^\gamma_\chi f(z,u) = \int_{\RR} f(t, \gamma(t, z,u))\chi(t,z,u)dt.
\]
\end{definition}
\begin{theorem}\label{averagingOp}
Let $\gamma\colon\RR\times\RR^2\times\RR^d\to\RR$ be smooth and suppose that $\gamma$ satisfies the cinematic curvature condition \eqref{soggeLocalMultiParam} at the origin. Then the following holds for all $\eps>0$ and all smooth bump functions $\chi\colon\RR\times\RR^2\times\RR^d\to\RR$ with support contained in a sufficiently small neighbourhood of the origin.

Let $f\colon \RR^2\to\CC$ and let $P_kf$ denote the Littlewood-Paley projection of $f$ to frequencies of scale $\sim 2^k$. Then for all $p\ge (d+1)(d+2)$, we have
\begin{equation}\label{260920_TY}
\Vert A^\gamma_\chi P_k f\Vert_{L^p(\RR^2\times\RR^d)}\lesssim  2^{-(\frac{d+1}{p}-\eps) k}\Vert f\Vert_{L^p(\RR^2)},
\end{equation}
where the implicit constant depends only on $p,d,\gamma,\eps$, and $\chi$. 
\end{theorem}
An earlier version of this theorem appeared in the unpublished manuscript \cite{CGY23}. The argument presented here substantially simplifies that earlier proof. The key observation is that, once the maximal operator is written in the geometric form
$$
(t,\gamma(t,z,u)),
$$
the relevant decoupling structure is already visible in the original parametrization. In particular, it is unnecessary to first pass to the dual formulation and then perform a stationary-phase analysis, as is done in earlier approaches, for example in the work of \cite{MSS92}, \cite{LLO25} and \cite{CGY23}.

Next we introduce a smoothly localized version of the operator $M^\gamma$ from Definition \ref{averagingAndMaximalOperators}. 
\begin{definition}\label{mollifiedMGamma}
Let $\gamma\colon\RR\times\RR^2\times\RR^d\to\RR$ be smooth, and let $\chi\colon \RR\times\RR^2\times\RR^d\to\RR$ be a smooth bump function supported near the origin. Let $f\colon\RR^2\to\CC$. For $z\in\RR^2$, define the maximal operator
\[
M^\gamma_\chi f(z) = \sup_u |A^\gamma_\chi f(z,u)|.
\]
\end{definition}
Note that
\[
\Vert M^\gamma_\chi  P_kf(z)\Vert_{L^p_{z}} =\Vert A^\gamma_\chi P_kf(z,u)\Vert_{L^p_{z}(L^\infty_u)}, 
\]
where $L^p_{z}$ denotes $L^p(\RR^2)$ in the variable $z\in\RR^2$, and $L^\infty_u$ denotes $L^\infty(\RR^d)$ in the variable $u$. Thus
by Sobolev embedding, Theorem \ref{averagingOp} has the following consequence.

\begin{corollary}\label{maximalFnBdsCor}
Let $\gamma\colon\RR\times\RR^2\times\RR^d\to\RR$ be smooth and suppose that $\gamma$ satisfies the cinematic curvature condition \eqref{soggeLocalMultiParam} at the origin. Then the following holds for all $\eps>0$ and all smooth bump functions $\chi\colon\RR\times\RR^2\times\RR^d\to \R$ with support contained in a sufficiently small neighbourhood of the origin.

Let $f\colon \RR^2\to\CC$ and let $P_kf$ denote the Littlewood-Paley projection of $f$ to frequencies of scale $\sim 2^k$. Then for all $p\ge(d+1)(d+2)$, we have
\begin{equation}\label{MGammaChiHighFrequencyDecay}
\Vert M^\gamma_\chi P_k f\Vert_{L^p(\RR^2)}\lesssim  2^{-(\frac{1}{p}-\eps) k}\Vert f\Vert_{L^p(\RR^2)},
\end{equation}
where the implicit constant depends only on $p,d,\gamma, \eps$, and $\chi$. 
\end{corollary}

In Section \ref{interpolationSubSection} we will show how Corollary \ref{maximalFnBdsCor} can be combined with a geometric argument by the sixth author to obtain Theorem \ref{maximalFnBound}.


\subsection{Reductions}

    We fix a smooth bump function $\chi$ and a smooth curve $\gamma$ satisfying the  cinematic curvature condition, denoting $A_{\chi}^\gamma$ simply by $A$. Denote
    \begin{align}
       & \lambda=2^k,\\
       & \widehat{g}(\bxi):=\lambda^2 \widehat{f}(\lambda \bxi),
    \end{align}
    where $\bxi=(\xi_1,\xi_2)$. Then it's equivalent to prove
    \begin{align}\label{Aim}
         \Vert A_{\lambda}g \Vert_{L^p(\R^2\times \R^d)}  \lesim_{\varepsilon} \lambda^{-\frac{d+3}{p} +\varepsilon} \Vert g \Vert_{L^p(\R^2)}
    \end{align}
    for all $g$ with $\supp\, \widehat{g} \subset \{\bxi: |\bxi |\simeq 1 \}$ and $p\ge (d+1)(d+2)$, where we denote
    \begin{align}
        A_\lambda g(z,u):=\int_{\R^2} \int_{\R} \widehat{g}(\bxi) e^{i\lambda\Phi(t,z,u;\bxi)} \chi(t,z,u)dtd\bxi 
    \end{align}
    and
    \begin{align}
        \Phi(t,z,u;\bxi)=t\xi_1+\gamma(t,z,u) \xi_2.
    \end{align}
    We can decompose $A_{\lambda}g$ as
    \begin{align}
        A_{\lambda}g(z,u)=A_{\lambda}^{Main}g(z,u)+A_{\lambda}^{E}g(z,u),
    \end{align}
    where 
    \begin{align}
        A_{\lambda}^{Main}g(z,u):=\int_{\R^2} \int_{\R} \widehat{g}(\bxi) e^{i\lambda\Phi(t,z,u;\bxi)} \chi(t,z,u)\beta(\xi_2) dtd\bxi
    \end{align}
    and $\beta(\xi_2)$ is a smooth bump function supported on $|\xi_2|\simeq 1$. If $|\xi_2|\ll 1$, then $|\bxi|\simeq 1$ implies $|\xi_1|\simeq 1$, hence
    \begin{align}
        |\partial_t \Phi(t,z,u;\bxi)|\simeq |\xi_1|\simeq 1.
    \end{align}
    By integration by parts, we have 
    \begin{align}
        |m_\lambda(z,u;\bxi)|,\ |\partial_\xi^\alpha m_\lambda(z,u;\bxi)|\lesssim_N \lambda^{-N}
    \end{align}
    for all large integer $N>0$ and all multi-indices $\alpha\in \mathbb{Z}_+^2$, where 
    \begin{align}
        m_\lambda(z,u;\bxi):= \int_{\R}e^{i\lambda\Phi(t,z,u;\bxi)} \chi(t,z,u)dt.
    \end{align}
    Then it's easy to see for any $N>0$,
    \begin{align}
        \Vert A_{\lambda}^Eg \Vert_{L^p(\R^2\times \R^d)}  \lesssim_N \lambda^{-N} \Vert g \Vert_{L^p(\R^2)}.
    \end{align}
Denote
\begin{align}
    h_\lambda(t,\xi_2):=\int_{\R} \widehat{g}(\xi_1,\xi_2)e^{i\lambda t\xi_1} d\xi_1,
\end{align}
now we write the main term $A_{\lambda}^{Main}g$ as 
\begin{align}
    A_{\lambda}^{Main}g(z,u)=\int_{\R^2} e^{i\lambda \gamma(t,z,u) \xi_2} h_\lambda(t,\xi_2) \chi(t,z,u)\beta(\xi_2) dtd\xi_2.
\end{align}

\subsection{Variable coefficient decoupling  at degree \texorpdfstring{$d$}{}}

We will use the following decoupling inequality. The case $d=1$ is a special case of Theorem 1.4 of \cite{BHS20}. The general
case $d\ge 2$ can be proven similarly, by combining the bootstrapping argument in Pramanik and Seeger \cite{PS07} and the decoupling inequalities of Bourgain, Demeter and Guth \cite{BDG16}. For a detailed proof, see \cite{LLO25}, Theorem 3.2 and Appendix A (NB: the analogue of ``$d$'' in \cite{LLO25}  is $d+1$ in our setup).
\begin{lemma}\label{d-decoupling}
Denote
\begin{align}
    \mathbb{D}= (-1,1)\times  \mathbb{B}^{d+2}(0,2)\times (1/2,2).
\end{align}
Let $\Psi: (-1,1) \times \mathbb{B}^{d+2}(0,2) \to \mathbb{R} $ be a smooth function and $A$ be a smooth function with $\supp\, A \subset \mathbb{D}$. Let $\lambda \ge 1$, $w\in \mathbb{B}^{d+2}(0,2)$, and $g$ be a function on $(-1,1)\times (1/2,2)$. We consider the extension operator
\begin{align}
    E_{\lambda} g(w):=\int_{\R^2}  e^{i\lambda \eta \Psi(t,w)} A(t,w;\eta)  g(t,\eta) dt d\eta .
\end{align}
If we have the non-degenerate condition:
\begin{align}\label{260719le1-17}
    \rank D_w (\mathcal{T}(\Psi))= d+2,
\end{align}
where 
\begin{align}
    \mathcal{T}(\Psi):= (\Psi(t,w),\partial_t\Psi(t,w),\cdots,\partial_t^{d+1}\Psi(t,w)),\label{eqn:rank_Lemma_TY}
\end{align}
then for any $\varepsilon>0$ and $N>0$ and for any $p\ge (d+1)(d+2)$, we have
\begin{align}
    \Vert E_{\lambda}g \Vert_{L^p} \lesim_\varepsilon \lambda^{\frac{1}{d+1}(1-\frac{1}{p}-\frac{(d+1)(d+2)}{2p})+\varepsilon}\bigg(\sum_{|I_d|=\lambda^{-\frac{1}{d+1}}} \Vert E_{\lambda}g_{I_d} \Vert_{L^p(\R^2\times \R^d)}^p \bigg)^{\frac{1}{p}} +\lambda^{-N} \Vert g\Vert_2.\label{eqn:decoupling_TY}
\end{align}
Here the sum is over $\{I_d\}$, a partition of $\R$, and $g_{I_d}(t,\eta):= g(t,\eta)\chi_{I_d}(t)$ with $\chi_{I_d}(t)$ a smooth cut-off of $I_d$.

\end{lemma}

To use this decoupling estimate, we set $w=(z,u)$, $\eta=\xi_2$, $A(t,w;\eta)=\chi(t,z,u)\beta(\xi_2)$ and $\Psi(t,w)=\gamma(t,z,u)$. Then  \eqref{260719le1-17} is exactly the cinematic curvature condition for $\gamma$, i.e.
\begin{equation}
\det \left(\begin{array}{ccccc}
\partial_{z_1}\gamma & \partial_{z_2} \gamma & \partial_{u_1} \gamma& \ldots& \partial_{u_d} \gamma\\
\partial_{z_1}\partial_t \gamma & \partial_{z_2}\partial_t \gamma & \partial_{u_1}\partial_t \gamma& \ldots&  \partial_{u_d}\partial_t \gamma\\
\partial_{z_1}\partial_t^2 \gamma & \partial_{z_2}\partial_t^2 \gamma & \partial_{u_1}\partial_t^2 \gamma& \ldots&  \partial_{u_d}\partial_t^2 \gamma\\
\vdots & \vdots & \vdots & \ddots & \vdots \\
\partial_{z_1}\partial_t^{d+1} \gamma & \partial_{z_2}\partial_t^{d+1} \gamma & \partial_{u_1}\partial_t^{d+1} \gamma& \ldots&  \partial_{u_d}\partial_t^{d+1} \gamma
\end{array}\right) \neq 0.
\end{equation}
Let $\{I_d\}$ be a family of intervals covering the support of $t$, each with length $\lambda^{-\frac{1}{d+1}}$. Then Lemma \ref{d-decoupling} yields
\begin{align}\label{260719le1-20}
   & \Vert A_{\lambda}^{Main}g \Vert_{L^p(\R^2\times \R^d)} \lesim_{\varepsilon}\notag \\
    &\lambda^{\frac{1}{d+1}(1-\frac{1}{p}-\frac{(d+1)(d+2)}{2p})+\varepsilon}\bigg(\sum_{|I_d|=\lambda^{-\frac{1}{d+1}}} \Vert A_{\lambda,I_d}^{Main}g \Vert_{L^p(\R^2\times \R^d)}^p \bigg)^{\frac{1}{p}} +\lambda^{-N} \Vert g\Vert_2,
\end{align}
where 
\begin{align}
    A_{\lambda,I_d}^{Main}g(z,u):=\int_{\R^2} e^{i\lambda \gamma(t,z,u)\xi_2} h_{\lambda,I_d}(t,\xi_2) \chi(t,z,u)\beta(\xi_2) dtd\xi_2,
\end{align}
and $h_{\lambda,I_d}(t,\xi_2)$ denotes $h_{\lambda}(t,\xi_2)$ with $t$ restricted on $I_d$ smoothly. Here $\chi$ and $\beta$ may be different cut-off functions but they still work for the following argument. By the rapid decay of the kernel, we may first localize \(g\) to a ball of radius \(O(\lambda)\), at the cost of an error \(O_N(\lambda^{-N})\|g\|_{L^p}\). For the localized function, H\"older's inequality gives
\[
\lambda^{-N}\|g\|_2
\lesssim
\lambda^{-N+1-\frac{2}{p}}\|g\|_p.
\]
Since \(N\) is arbitrary, this term can be absorbed into the right-hand side of \eqref{Aim}, and will be omitted hereafter.

\subsection{Variable coefficient decoupling  at lower degree}

For each fixed interval $I_d$, after a permutation of the $(z,u)$, we may assume without loss of generality that
\begin{equation}\label{cc(d-1)}
\det \left(\begin{array}{ccccc}
\partial_{z_1}\gamma & \partial_{z_2} \gamma & \partial_{u_1} \gamma& \ldots& \partial_{u_{d-1}} \gamma\\
\partial_{z_1}\partial_t \gamma & \partial_{z_2}\partial_t \gamma & \partial_{u_1}\partial_t \gamma& \ldots&  \partial_{u_{d-1}}\partial_t \gamma\\
\partial_{z_1}\partial_t^2 \gamma & \partial_{z_2}\partial_t^2 \gamma & \partial_{u_1}\partial_t^2 \gamma& \ldots&  \partial_{u_{d-1}}\partial_t^2 \gamma\\
\vdots & \vdots & \vdots & \ddots & \vdots \\
\partial_{z_1}\partial_t^{d} \gamma & \partial_{z_2}\partial_t^{d} \gamma & \partial_{u_1}\partial_t^{d} \gamma& \ldots&  \partial_{u_{d-1}}\partial_t^{d} \gamma
\end{array}\right) \neq 0,
\end{equation}
for all $t\in I_d$. This is because $z$ and $u$ play the same role in this argument; we don’t need to distinguish between them. By Taylor expansion, we have
\begin{align}
    \gamma(t,z,u)=\gamma(c_d,z,u)&+\partial_t \gamma(c_d,z,u)(t-c_d)+\cdots \notag\\&+\partial_t^d \gamma(c_d,z,u)\frac{(t-c_d)^d}{d!} +O(|t-c_d|^{d+1}).
\end{align}
Since $|I_d|=\lambda^{-\frac{1}{d+1}}$, we can ignore the $|t-c_d|^{d+1}=O(1)$ term, i.e., $A_{\lambda,I_d}^{Main}g$ can be reduced to
\begin{align}
    \int_{\R^2} e^{i\lambda \gamma_{d-1}(t,z,u)\xi_2} h_{\lambda,I_d}(t,\xi_2) \chi(t,z,u)\beta(\xi_2) dtd\xi_2,
\end{align}
where 
\begin{align}
    \gamma_{d-1}(t,z,u)= \gamma(c_d,z,u)+\partial_t \gamma(c_d,z,u)(t-c_d)+\cdots \notag+\partial_t^d \gamma(c_d,z,u)\frac{(t-c_d)^d}{d!}.
\end{align}
For each fixed $u_d$, \eqref{cc(d-1)} says that the phase function $\gamma_{d-1}(t,z,u)$ satisfies \eqref{260719le1-17} with $w= (z,u_1,\cdots,u_{d-1})$. There is a standard method to use decoupling: we translate and rescale $I_d$ to the unit interval and use Lemma \ref{d-decoupling} with $d+2$ replaced by $d+1$, and then we transform the unit interval back to $I_d$. Then we will obtain that for $p\ge d(d+1)$,
\begin{align}\label{260719le1-25}
    \Vert A_{\lambda,I_d}^{Main}g(\cdot,u_d) \Vert_{L^p(\R^2\times \R^{d-1})} \lesssim_{\varepsilon}&  \lambda^{(\frac{1}{d}-\frac{1}{d+1})(1-\frac{1}{p}-\frac{d(d+1)}{2p})+\varepsilon} \cdot\notag \\ &\bigg(\sum_{\underset{|I_{d-1}|=\lambda^{-\frac{1}{d}}}{ I_{d-1}\subset I_d}} \Vert A_{\lambda,I_{d-1}}^{Main}g(\cdot,u_d) \Vert_{L^p(\R^2\times \R^{d-1})}^p\bigg)^{ \frac{1}{p}}.
\end{align}
Since the implicit decoupling constant relies on the determinant in \eqref{cc(d-1)}, which is uniform in $u_d$, we can integrate \eqref{260719le1-25} with respect to $u_d$ and obtain
\begin{align}
    \Vert A_{\lambda,I_d}^{Main}g \Vert_{L^p(\R^2\times \R^d)} \lesssim_{\varepsilon}   \lambda^{(\frac{1}{d}-\frac{1}{d+1})(1-\frac{1}{p}-\frac{d(d+1)}{2p})+\varepsilon} \bigg( \sum_{\underset{|I_{d-1}|=\lambda^{-\frac{1}{d}}}{ I_{d-1}\subset I_d}} \Vert A_{\lambda,I_{d-1}}^{Main}g \Vert_{L^p(\R^2\times \R^d)}^p\bigg)^{\frac{1}{p}}.
\end{align}
By the same token, we have
\begin{align}
    \Vert A_{\lambda,I_j}^{Main}g \Vert_{L^p(\R^2\times \R^d)} \lesssim_{\varepsilon}   \lambda^{(\frac{1}{j}-\frac{1}{j+1})(1-\frac{1}{p}-\frac{j(j+1)}{2p})+\varepsilon}  \bigg( \sum_{\underset{|I_{j-1}|=\lambda^{-\frac{1}{j}}}{ I_{j-1}\subset I_j}} \Vert A_{\lambda,I_{j-1}}^{Main}g \Vert_{L^p(\R^2\times \R^d)}^p\bigg)^{\frac{1}{p}}
\end{align}
for $j=2,\cdots,d$. Combining these $d-1$ estimates with \eqref{260719le1-20} implies
\begin{align}
     \Vert A_{\lambda}^{Main}g \Vert_{L^p(\R^2\times \R^d)} \lesim_{\varepsilon} \lambda^{\frac12-\frac{d+1}{p}+\varepsilon}\bigg(\sum_{|I_1|=\lambda^{-\frac{1}{2}}} \Vert A_{\lambda,I_1}^{Main}g \Vert_{L^p(\R^2\times \R^d)}^p \bigg)^{\frac{1}{p}}.
\end{align}
To prove \eqref{Aim}, it suffices to prove
\begin{align}\label{Aim2}
    \bigg(\sum_{|I_1|=\lambda^{-\frac{1}{2}}} \Vert A_{\lambda,I_1}^{Main}g \Vert_{L^p(\R^2\times \R^d)}^p \bigg)^{\frac{1}{p}}\lesssim \lambda^{-\frac12-\frac{2}{p}}\|g\|_{L^p(\R^2)}.
\end{align}

\subsection{Proof of \texorpdfstring{\eqref{Aim2}}{}}

Recall that
\begin{align}
    A_{\lambda,I_1}^{Main}g(z,u)= \int_{I_1} \int_{\R^2} \widehat{g}(\bxi) e^{i\lambda(t\xi_1+ \gamma(t,z,u)\xi_2)} \chi(t,z,u)\beta(\bxi) d\bxi dt,
\end{align}
where we abuse the notation $\beta$. We fix $z$ and $u$ and apply Plancherel to the integral with respect to $\xi$, then
\begin{align}
    \left|\int_{\R^2} \widehat{g}(\bxi) e^{i\lambda(t\xi_1+ \gamma(t,z,u)\xi_2)} \beta(\bxi) d\bxi\right| &=\left|\int_{\R^2} g(v) \hat{\beta}(v_1-\lambda t,v_2-\lambda \gamma(t,z,u)) dv\right| \notag\\
    &\le\|g\|_{L^\infty(\R^2)}\|\hat{\beta}\|_{L^1(\R^2)} \notag\\
    &\lesssim \|g\|_{L^\infty(\R^2)}.
\end{align}
Since $|I_1|=\lambda^{-\frac12}$, we obtain \eqref{Aim2} for $p=\infty$,
\begin{align}\label{260719le1-32}
    \sup_{|I_1|=\lambda^{-\frac{1}{2}}} \Vert A_{\lambda,I_1}^{Main}g \Vert_{L^\infty (\R^2\times \R^d)}\lesssim \lambda^{-\frac12}\|g\|_{L^\infty(\R^2)}.
\end{align}
The rest is to prove \eqref{Aim2} for $p=2$. We recall that 
\begin{align}
    A_{\lambda,I_1}^{Main}g(z,u)=\int_{\R^2} e^{i\lambda \gamma(t,z,u)\xi_2} h_{\lambda,I_1}(t,\xi_2) \chi(t,z,u)\beta(\xi_2) dtd\xi_2.
\end{align}
Let $\Psi(t,z,u,\xi_2)=\gamma(t,z,u)\xi_2$. Then a direct calculation shows
\begin{equation}
    \nabla_{(z,u)}\nabla_{(t,\xi_2)}\Psi =\left(
    \begin{array}{ccccc}
\partial_{z_1}\partial_t\gamma\cdot\xi_2 & \partial_{z_2}\partial_t \gamma\cdot\xi_2 & \partial_{u_1}\partial_t \gamma\cdot\xi_2& \ldots& \partial_{u_{d}}\partial_t \gamma\cdot\xi_2\\
\partial_{z_1} \gamma & \partial_{z_2} \gamma & \partial_{u_1} \gamma& \ldots&  \partial_{u_{d}} \gamma
\end{array}\right).
\end{equation}
Since $|\xi_2|\simeq 1$, the cinematic curvature condition indicates that there exists a $2\times 2$ non-degenerate sub-matrix of  $\nabla_{(z,u)}\nabla_{(t,\xi_2)}\Psi$. After a permutation of the $(z,u)$, we may assume without loss of generality that
\begin{align}
    \det \left( \begin{array}{ccccc}
\partial_{z_1}\partial_t\gamma\cdot\xi_2 & \partial_{z_2}\partial_t \gamma\cdot\xi_2 \\
\partial_{z_1} \gamma & \partial_{z_2} \gamma 
\end{array}\right) \neq 0.
\end{align}
So for each $u$, H\"ormander's $L^2$ theorem implies
\begin{align}
    \Vert A_{\lambda,I_1}^{Main}g (\cdot,u)\Vert_{L^2 (\R^2)}\lesssim \lambda^{-1}\|h_{\lambda,I_1}\|_{L^2(\R^2)}.
\end{align}
The constant is uniform in $u$, hence an integration with respect to $u$ yields
\begin{align}
    \Vert A_{\lambda,I_1}^{Main}g \Vert_{L^2 (\R^2\times\R^d)}\lesssim \lambda^{-1}\|h_{\lambda,I_1}\|_{L^2(\R^2)}.
\end{align}
Then summing over $I_1$ implies
\begin{align}
    \bigg(\sum_{|I_1|=\lambda^{-\frac{1}{2}}} \Vert A_{\lambda,I_1}^{Main}g \Vert_{L^2(\R^2\times \R^d)}^2 \bigg)^{\frac{1}{2}}\lesssim\lambda^{-1}\|h_\lambda\|_{L^2}.
\end{align}
Recall that $h_\lambda$ denotes the $\xi_1$-Fourier inverse of $\hat{g}$ at $t\lambda$, 
\begin{align}
     h_\lambda(t,\xi_2):=\int_{\R} \widehat{g}(\xi_1,\xi_2)e^{i\lambda t\xi_1} d\xi_1.
\end{align}
So we have 
\begin{align}
    \|h_\lambda\|_{L^2(\R^2)}= \lambda^{-\frac12} \|g\|_{L^2(\R^2)}.
\end{align}
As a result, we obtain \eqref{Aim2} for $p=2$, and an interpolation ends the proof.

\subsection{From Corollary \ref{maximalFnBdsCor} to Theorem \ref{maximalFnBound} - interpolation with a geometric estimate}\label{interpolationSubSection} 
In this section we will prove Theorem \ref{maximalFnBound}. We will do so by combining Corollary \ref{maximalFnBdsCor} with the following result by the sixth author:
\begin{theorem}\label{ZahlThm}
Let $I\subset\RR$ be a closed interval, let $W_0\subset\RR^2\times\RR^d$ be open, let $\gamma_0\colon I\times W_0 \to\RR$ be smooth. Suppose that $\gamma_0$ satisfies the multiparameter cinematic curvature condition \eqref{soggeLocalMultiParam} on $I\times W_0$. 

Let $W\subset W_0$ be a closed cube, let $\gamma$ be the restriction of $\gamma_0$ to $I\times W$, and let $M^\gamma$ be the maximal operator associated to $\gamma$, in the sense of Definition \ref{averagingAndMaximalOperators}. Let $\eps>0$. Then
\begin{equation}\label{LdBdZahl}
\Vert M^\gamma (P_k f ) \Vert_{L^{d+2}(\RR^2)}\leq C_{\gamma,\eps}2^{k\eps}\Vert f\Vert_{L^{d+2}(\RR^2)}.
\end{equation}
\end{theorem}
Theorem \ref{ZahlThm} is an immediate consequence of the following slightly stronger result: suppose that $W = I_1\times I_2\times C\subset W_0$ is a closed cube, and let
\begin{align}
    {M}_1^\gamma f(z_2) = \sup_{z_1\in I_1,\, u\in C} \left|\int_I f(t, \gamma(t, z_1,z_2, u))dt\right|.
\end{align}
Then we have
\begin{equation}\label{oneDimSup}
\Vert {M}_1^\gamma (P_kf) \Vert_{L^{d+2}(\RR)}\leq C_{\gamma,\eps}2^{k\eps}\Vert f\Vert_{L^{d+2}(\RR^2)}.
\end{equation}
Inequality \eqref{oneDimSup} is essentially Theorem 1.5 from \cite{Zah26}. We will briefly summarize the similarities and differences between these two statements. 
\begin{itemize}
    \item In the language of \cite{Zah26}, $M_1^\gamma$ is an $s=d+1$ parameter maximal function associated to an $m = d+2$ dimensional family of cinematic curves (the definition of cinematic curves from \cite{Zah26} coincides with the cinematic curvature condition \eqref{soggeLocalMultiParam}).
    \item  Note that \cite{Zah26} defines a ``transversality condition'' (\cite{Zah26}, Definition 1.3) which is sometimes difficult to verify. However in our setting this transversality condition is vacuously satisfied, since $s = m-1$. 

    \item The maximal function from \cite[Theorem 1.5]{Zah26} is taken over $\delta$-thickened curves, whereas Inequality \eqref{oneDimSup} refers to a maximal function taken over genuine curves, but applied to functions that have been blurred-out at scale $2^{-k}$ by the Littlewood–Paley projection $P_k$. This is harmless, because the maximal average $M_1^\gamma(P_k f)$ can be controlled by a dyadic sum of maximal functions over $\delta$-thickened curves, as $\delta$ ranges over $2^{-\ell},$ $\ell=1,\ldots,k$.
\end{itemize}

We are now ready to prove Theorem \ref{maximalFnBound}.

\begin{proof}[Proof of Theorem \ref{maximalFnBound}]
We first consider the case $d\geq 2$. Since $I\times W'$ is compact, it suffices to consider the case where $I$ is a short interval and $W'\subset\RR^2\times\RR^d$ is a small closed disk. For notational convenience we will suppose that $I$ and $W'$ are centered at the origin.

Extend $\gamma$ smoothly to a neighbourhood of $I\times W'$ if necessary. Select a nonnegative bump function $\chi$ that is identically $1$ on $I\times W'$, and whose support is sufficiently small so that Corollary \ref{maximalFnBdsCor} applies. For every function $f\colon\RR^2\to\CC$, we have
\begin{equation}\label{smoothDomination}
M^\gamma f(z)\leq M^\gamma_\chi(|f|)(z).
\end{equation}

The argument used to obtain \eqref{LdBdZahl} from the thickened-curve estimate also gives, for every $\eps>0$,
\[
\Vert M^\gamma_\chi(P_kf)\Vert_{L^{d+2}(\RR^2)}
\lesssim_{\gamma,\chi,\eps}
2^{k\eps}\Vert f\Vert_{L^{d+2}(\RR^2)}.
\]
Combining this estimate with Corollary \ref{maximalFnBdsCor} and interpolation, we conclude that for each $d+2<p<\infty$, there exists $\eta>0$ so that
\begin{equation}\label{highFrequencyDecayd+2}
\Vert M^\gamma_\chi(P_kf)\Vert_{L^p(\RR^2)}
\leq C_{\gamma,\chi,p}2^{-k\eta}\Vert f\Vert_{L^p(\RR^2)}.
\end{equation}
Summing over $k$ and treating the low frequencies separately, we obtain
\[
\Vert M^\gamma_\chi f\Vert_{L^p(\RR^2)}
\lesssim_{\gamma,\chi,p}\Vert f\Vert_{L^p(\RR^2)}.
\]
Applying this estimate to $|f|$ and using \eqref{smoothDomination}, we conclude that
\[
\Vert M^\gamma f\Vert_{L^p(\RR^2)}
\lesssim_{\gamma,\chi,p}\Vert f\Vert_{L^p(\RR^2)}.
\]
The case $p=\infty$ is immediate.
If $d=1$, the result follows from \cite{MSS92}.
\end{proof}


\section{The case $n=1$: Nonlinear Kakeya-type sets}\label{n=1Case}
Recall that in the introduction, we defined $W\subset\RR^n\times\RR^d$ and $\gamma\colon I\times W\to\RR$. In this section we consider the case $n=1$. This is sometimes referred to as the ``Kakeya'' setting. Our main tool is the following result due to Wisewell. We will state a special case of this theorem---the version in \cite{Wis04} has slightly weaker regularity hypotheses on the function $f$.

\begin{theorem}[\cite{Wis04} Theorem 2]\label{WisewellThm}
Let $0\leq p\leq n-e\leq q$ be integers, with $e<n$. There exists a function $\phi\colon\RR^p\to\RR^q$ with the following property.

Let $f(y,\omega,t)\colon\RR^p\times\RR^q\times\RR^e\to\RR^{n-e}$ have Lipschitz derivative, and suppose that the Jacobian $\frac{df}{d\omega}$ has full rank $n-e$. Then the set
\[
E_f = \bigcup_{y\in\RR^p}\operatorname{graph} f(y, \phi(y),\cdot)
\]
has measure 0.
\end{theorem}
The proof in \cite{Wis04} is constructive, and in particular the function $\phi$ from Theorem \ref{WisewellThm} is measurable. We will focus on the planar, one-parameter case $p=q=e=1$ and $n=2$. For our applications, we will consider functions $f(y,\omega,t)$ whose domain is $[-1,1]^3$ rather than $\RR^3$. However, by composing $f$ with $(y,\omega,t)\mapsto (\frac{2}{\pi}\arctan y,\ \frac{2}{\pi}\arctan \omega,\frac{2}{\pi}\arctan t)$, we have the following ``local'' version of Theorem \ref{WisewellThm}. 

\begin{corollary}\label{localWisewellThm}
There exists a positive measure set $F\subset[-1,1]$ and a function $\phi\colon F\to[-1,1]$ with the following property. Let $f(y,\omega,t)\colon [-1,1]^3 \to\RR$ have Lipschitz derivative, and suppose that $\frac{df}{d\omega}\neq 0$ on $[-1,1]^3$. Then the set
\begin{equation}\label{setEfromWisewell}
E=\bigcup_{y\in F}\operatorname{graph} f(y, \phi(y),\cdot)
\end{equation}
has measure 0.
\end{corollary}
Corollary \ref{localWisewellThm} allows us to handle the $n=1$ analogue of Theorem \ref{characterizationOfNikodymSets}.
\begin{proposition}
Let $I\subset\RR$ be a bounded closed interval. Let $W\subset\RR\times\RR$ be open. Let $\gamma\colon I\times W\to\RR$ have Lipschitz derivative. Let $\Gamma$ be the family of plane curves associated to the defining function $\gamma$. Suppose that $\partial_{u_1}\gamma$ does not vanish identically on $I\times W$. Then $\Gamma$ admits Nikodym sets. 
\end{proposition}


\section{Falconer's prescribed projection theorem}\label{curveReuseDaviesSection}
In this section we will state a nonlinear variant of Falconer's prescribed projection theorem \cite{Fal86}, which was recently proved by the first author, Fraser, McDonald, Meyer, and Taylor \cite{CFMMT}. Before stating the theorem, we will need several definitions.

Let $D\subset\RR^3$ be open and let $\phi(\alpha,z)\colon D\to\RR$. We define $\phi^\alpha(z) = \phi(\alpha,z)$, $D^\alpha=\{z\colon (\alpha,z)\in D\}$, and $\mathcal{A}=\{\alpha\colon D^\alpha\neq\emptyset\}$. We define $\nabla_z\phi(\alpha,z) = \nabla \phi^\alpha(z)=(\partial_{z_1}\phi^\alpha(z),\partial_{z_2}\phi^\alpha(z))$ to be the gradient in the $z$ variables.

\begin{definition}\label{defnMonotonicity}
We say a function $g\colon X\to S^1$ has \emph{restricted range} if $g(X)$ is contained in $\{(\cos\theta,\sin\theta) \colon \theta \in [c,c+\pi)\}$ for some $c\in\RR$.

Let $X\subset\RR$ and let $g\colon X\to S^1$ have restricted range. We say that $g$ is strictly increasing (resp.~decreasing, monotone) if the corresponding map $g\colon X\to [c,c+\pi)$ is strictly increasing (resp.~decreasing, monotone).
\end{definition}
\begin{remark}
Note that Definition \ref{defnMonotonicity} is stronger than the corresponding definition from \cite{CFMMT}, i.e.~every function that satisfies Definition \ref{defnMonotonicity} is strictly monotone in the sense of \cite{CFMMT}, but the converse does not hold.
\end{remark}

\begin{definition}\label{angleMonotonicityDefn}
We say $\phi$ is \emph{non-stationary} if $\nabla\phi^\alpha\neq 0$ for all $(\alpha,z)\in D$. We say $\phi$ satisfies the \emph{angle monotonicity condition} if $\nabla \phi^\alpha(z)/\Vert \nabla \phi^\alpha(z)\Vert\colon D\to S^1$ has restricted range, and for each $z$, the function $\alpha\mapsto \nabla \phi^\alpha(z)/\Vert \nabla \phi^\alpha(z)\Vert$ is strictly monotone, in the sense of Definition \ref{defnMonotonicity}.
\end{definition}

The following is \cite[Theorem 1.2]{CFMMT}.
\begin{theorem}[\cite{CFMMT}, Theorem 1.2]\label{CFMMT_thm}
Let $D\subset\RR^3$ be open, and let the sets $\mathcal{A}$ and $D^\alpha,\ \alpha\in\mathcal{A}$ be as defined above. Suppose that for each $\alpha \in \mathcal{A}$, $D^\alpha$ is convex. Suppose that $\phi\colon D\to\RR$ is $C^1$, non-stationary, and satisfies the angle monotonicity condition. Suppose furthermore that $\phi$ and $\nabla_z\phi$ both extend continuously to the closure $\overline D$, and $\nabla_z \phi$ is nonzero on $\overline D$. 

For each $\alpha\in\mathcal{A}$, let $F^\alpha\subset \phi^\alpha(D^\alpha)$. Suppose that $\bigcup_{\alpha\in \mathcal{A}}(\{\alpha\}\times F^\alpha)$ is measurable. Then there exists a Borel set $G\subset\RR^2$ such that for a.e.~$\alpha\in\mathcal{A}$, we have
\[
\phi^\alpha(G\cap D^\alpha)\supset F^\alpha\qquad\textrm{and}\qquad |\phi^\alpha(G\cap D^\alpha) \backslash F^\alpha|=0.
\]
\end{theorem}

\begin{proposition}\label{prescribedProjectionsTwoDirections}
Let $U\subset\RR^2$ be a connected, open set. Let $I_1,I_2\subset\RR$ be open intervals, and let $\phi\colon I_1\times U\to\RR$, $\psi\colon I_2\times U\to\RR$ be analytic. Suppose that $\nabla \phi^\alpha(z)\neq 0$ on $I_1\times U$, and $\partial_\alpha \big(\frac{\nabla \phi^\alpha(z)}{\Vert \nabla \phi^\alpha(z)\Vert}\big)\not\equiv 0$ on $I_1\times U$, and similarly for $\psi$.

Then there exists a Borel set $G\subset U$ and an open interval $I_1'\subset I_1$ so that
\begin{equation}\label{conclusionOfPropPrescribedProjectionsTwoDirections}
\int_{I_1'}|\phi^\alpha(G)|d\alpha = 0,\qquad\textrm{and}\qquad\int_{I_2}|\psi^\alpha(G)|d\alpha > 0.
\end{equation}
\end{proposition}
\begin{proof}
Let $X_1 = \{(\alpha,z) \in I_1 \times U : \partial_\alpha \big(\frac{\nabla \phi^\alpha(z)}{\Vert \nabla \phi^\alpha(z)\Vert}\big) \neq 0\}$. Note that if $(\alpha_0,z_0) \in X_1$, then there exists a neighbourhood of $(\alpha_0,z_0)$ on which $\alpha\mapsto\frac{\nabla \phi^\alpha(z)}{\Vert \nabla \phi^\alpha(z)\Vert}$ is strictly monotone (in the sense of Definition \ref{defnMonotonicity}) for all $z$ near $z_0$. A similar statement holds for $X_2 = \{(\alpha,z) \in I_2 \times U : \partial_\alpha \big(\frac{\nabla \psi^\alpha(z)}{\Vert \nabla \psi^\alpha(z)\Vert}\big) \neq 0\}$.

Let $\pi : \RR^3 \to \RR^2$ denote the projection to the last 2 variables. By analyticity, the set $X_1$ is open and has full measure in $I_1 \times U$, and hence $\pi(X_1)$ has full measure in $U$. Similarly, the set $X_2$ is open and has full measure in $I_2 \times U$, and hence $\pi(X_2)$ has full measure in $U$. 

Let $u \in \pi(X_1) \cap \pi(X_2)$ and select points $(\alpha_1,u)\in X_1$ and $(\alpha_2,u)\in X_2$ so that  $\frac{\nabla \phi^{\alpha_1}(u)}{\Vert \nabla \phi^{\alpha_1}(u)\Vert}\not\in \{\pm \frac{\nabla \psi^{\alpha_2}(u)}{\Vert \nabla \psi^{\alpha_2}(u)\Vert} \}$. We can do this since $\alpha\mapsto\frac{\nabla \psi^\alpha(u)}{\Vert \nabla \psi^\alpha(u)\Vert}$ is locally strictly monotone.

Since $(\alpha_1,u)$ is an interior point of $X_1$ and similarly for $(\alpha_2,u)$, if we select a sufficiently small ball $U'\subset U$ centered at $u$ and sufficiently small neighbourhoods $I_1'\subset I_1$ of $\alpha_1$ and $I_2'\subset I_2$ of $\alpha_2$, then we can ensure that $\overline{I_1'\times U'}\subset X_1$ and $\overline{I_2'\times U'}\subset X_2$. After shrinking the sets $I_1',I_2'$ and $U'$ if needed, we may suppose that the images of $\frac{\nabla \phi^\alpha(z)}{\Vert \nabla \phi^\alpha(z)\Vert}$ and $\frac{\nabla \psi^\alpha(z)}{\Vert \nabla \psi^\alpha(z)\Vert}$ are contained in disjoint circular arcs that have positive separation. Since $\frac{\nabla \phi^{\alpha_1}(u)}{\Vert \nabla \phi^{\alpha_1}(u)\Vert}\not\in \{\pm \frac{\nabla \psi^{\alpha_2}(u)}{\Vert \nabla \psi^{\alpha_2}(u)\Vert} \}$, we can also ensure that these two circular arcs are contained in $\{(\cos\theta,\sin\theta) \colon \theta \in [c,c+\pi)\}$ for some $c\in\RR$. In particular, the restriction of $\phi$  to $I_1'\times U'$ satisfies the conditions in the first paragraph of Theorem \ref{CFMMT_thm}, and similarly for the restriction of $\psi$ to $I_2'\times U'$.

We have that $\alpha\mapsto \frac{\nabla \phi^\alpha(z)}{\Vert \nabla \phi^\alpha(z)\Vert}$ is strictly increasing or decreasing for each $z\in U'$, and similarly for $\alpha\mapsto \frac{\nabla \psi^\alpha(z)}{\Vert \nabla \psi^\alpha(z)\Vert}$. 

Suppose that all points on the arc containing the image of $\frac{\nabla \phi}{\Vert \nabla \phi\Vert}$ are ``smaller'' (in the ordering coming from Definition \ref{defnMonotonicity}) than the points on the arc containing the image of $\frac{\nabla \psi^\alpha(z)}{\Vert \nabla \psi^\alpha(z)\Vert}$. Let $I_1' = (a,b)$. By pre-composing $\phi^\alpha(z)=\phi(\alpha,z)$ with a map of the form $\alpha\mapsto a + (b-a)\alpha$ or $\alpha\mapsto b-(b-a)\alpha$, we obtain a map $\tilde\phi \colon (0,1) \times U' \to \RR$ with the property that $\alpha\mapsto \frac{\nabla \tilde\phi^\alpha(z)}{\Vert \nabla \tilde \phi^\alpha(z)\Vert}$ is strictly increasing for each $z\in U'$. Perform a similar procedure to produce a map $\tilde\psi$ with domain $(2,3)\times U'$ so that $\alpha\mapsto \frac{\nabla \tilde\psi^\alpha(z)}{\Vert \nabla \tilde \psi^\alpha(z)\Vert}$ is strictly increasing for each $z\in U'$. 

If instead all points on the arc containing the image of $\frac{\nabla \phi}{\Vert \nabla \phi\Vert}$ are ``larger'' than the points  on the arc containing the image of $\frac{\nabla \psi^\alpha(z)}{\Vert \nabla \psi^\alpha(z)\Vert}$, then perform a similar procedure to obtain maps $\tilde\phi\colon (0,1)\times U'\to\RR$ and $\tilde\psi\colon (2,3)\times U'\to\RR$  with the property that $\alpha\mapsto \frac{\nabla \tilde\phi^\alpha(z)}{\Vert \nabla \tilde \phi^\alpha(z)\Vert}$ and $\alpha\mapsto \frac{\nabla \tilde\psi^\alpha(z)}{\Vert \nabla \tilde \psi^\alpha(z)\Vert}$ are strictly decreasing for each $z\in U'$.

Define
\[
\zeta(\alpha,z)=\left\{\begin{array}{ll}
\tilde\phi(\alpha, z),\ & (\alpha,z)\in (0,1)\times U',\\
\tilde\psi(\alpha, z),\ & (\alpha,z)\in (2,3)\times U'.\\
\end{array}
\right.
\]
We selected the maps $\tilde\phi$ and $\tilde\psi$ so that $\zeta$ satisfies the angle monotonicity condition. $U'$ is a ball, so it is convex. Since $\overline{I_1'\times U'} \subset I_1 \times U$ and $\overline{I_2'\times U'} \subset I_2 \times U$, both $\zeta$ and $\nabla_z \zeta$ have continuous extensions to the closure, and $\nabla_z\zeta \neq 0$ on the closure. Thus $\zeta$ satisfies the hypotheses of Theorem \ref{CFMMT_thm}. Select $F^\alpha = \tilde\psi^\alpha(U')$ for $\alpha\in (2,3)$. Since $\nabla_z \tilde\psi \neq 0$, the map $(\alpha,z) \mapsto (\alpha, \tilde\psi(\alpha,z))$ is an open map. Thus $\{(\alpha, p)\colon \alpha\in(2,3),\ p\in \tilde\psi^\alpha(U')\}$ (which is the image of $(2,3) \times U'$) is open and therefore measurable. Select $F^\alpha = \emptyset$ for $\alpha\in (0,1)$. Apply Theorem \ref{CFMMT_thm} to $\zeta\colon \big((0,1)\cup(2,3)\big)\times U'\to\RR$ with this choice of sets $\{F^\alpha\}$. Denote the output by $G$; abusing notation, we will replace $G$ with $G\cap U'$. Then $|\phi^\alpha(G)|=0$ for a.e.~$\alpha\in I_1'$, and $|\psi^\alpha(G)|>0$ for a.e.~$\alpha\in I_2'$. 
\end{proof}


\section{Tame geometry}\label{semiAnalyticGeomSection}
In this section, we will introduce some tools that will allow us to break sets of the form $\{x\colon f(x)=0\}$ or $\{x\colon f(x)<0\}$ (here $f$ is an analytic function) into simpler pieces. This will be a key ingredient in the proof of Proposition \ref{breakIntoConstantMpAndMc}, which we will prove at the end of this section. We begin with several definitions. 

\begin{definition}
A \emph{restricted analytic function} on $\RR^n$ is a function of the form
\[
g(x)=\left\{\begin{array}{ll} 
f(x),&\ x\in [0,1]^n,\\
0,& \textrm{otherwise}\end{array}\right.,
\]
where $f$ extends to an analytic function on a neighbourhood of $[0,1]^n$.
\end{definition}

\begin{definition}
A set $S\subset\RR^n$ is called a basic definable set (in the o-minimal structure $\RR_{\operatorname{an}}$) if it is of the form 
\begin{equation}\label{basicDefinableSet}
\{x\in\RR^n\colon f_1(x)=0,\ldots,f_k(x)=0,\ g_1(x)>0,\ldots,g_\ell(x)>0\},
\end{equation}
where $f_1,\ldots,f_k$ and $g_1,\ldots,g_\ell$ are polynomials (with real coefficients) or restricted analytic functions.
\end{definition}

\begin{definition}
 A set $S$ is called a \emph{definable set} (in the o-minimal structure $\RR_{\operatorname{an}}$) if it can be obtained from basic definable sets after applying finitely many operations of the following form: finite unions, finite intersections, complements, Cartesian products, coordinate projections. In particular, every semi-algebraic set is a definable set. A function $f\colon\RR^n\to\RR$ is called \emph{definable} if its graph is a definable set. 
\end{definition}

See \cite{Dries09} for an introduction to the theory of definable sets, and \cite{DMM94} and \cite[\S2.5, Example (4)]{DM96} for a discussion of $\RR_{\operatorname{an}}$.

\subsection{The analytic cell decomposition}
We recall the definition of an analytic cell from \cite[\S 4.1]{DM96}  (in that section the authors define a $C^p$ cell. In the subsequent discussion in \S4.2 they discuss the analytic case $p=\omega$). See also \cite[Chapter 3, \S2 ]{Dries09}. Here and in what follows, we are working in $\RR_{\operatorname{an}}$.
\begin{definition}\label{defnCell}
In what follows, let $(i_1,\ldots,i_m)$ be a sequence of 0s and 1s of length $m$. When $m=1$, an (analytic) $(0)$-cell is a point $\{z_1\}\subset \RR$. An (analytic) $(1)$-cell is an open interval $(a,b)\subset\RR$ (in this definition, $a=-\infty$ and/or $b=\infty$ are allowed). Now we define an analytic  $(i_1,\ldots,i_{m+1})$-cell by induction on $m$ as follows. Suppose that $(i_1,\ldots,i_m)$ cells are already defined. An analytic $(i_1,\ldots,i_m,0)$-cell is a graph $\operatorname{graph}(f)\subset\RR^{m+1}$ of a definable function $f\colon D\to\RR$, where $D\subset\RR^m$ is a $(i_1,\ldots,i_m)$-cell, and $f$ extends to an analytic function on a neighbourhood of $D$. An analytic $(i_1,\ldots,i_m,1)$-cell is a set of the form
\[
\{(z, y)\in D\times\RR\colon f(z)<y<g(z)\},
\]
where $D$ is a $(i_1,\ldots,i_m)$-cell, and one of the following holds
\begin{itemize}
    \item $f,g\colon D\to\RR$ are definable and analytic on a neighbourhood of $D$, and $f<g$ on $D$.
    \item $f$ as above and $g=\infty$,
    \item $g$ as above and $f=-\infty$,
    \item $f=-\infty$ and $g=\infty$, i.e. the cell is $D\times\RR$. 
\end{itemize}
\end{definition}
Note that a $(1,\ldots,1)$-cell in $\RR^m$ is an open set (in the Euclidean topology). All other cells have dimension at most $m-1$. We call a $(1,\ldots,1)$-cell an \emph{open} cell. It is straightforward to verify that if $D$ is an $(i_1,\ldots,i_m)$-cell and $\sum_{j=1}^m i_j=d$, then $\dim D=d$, and there is a definable, analytic bijection (with definable, analytic inverse) $G\colon (0,1)^d\to D$.

Note that in Definition \ref{defnCell}, the ordering of the coordinates $z_1,\ldots,z_n$ is important; a set that is a cell for the ordering $z_1,\ldots,z_n$ might cease to be a cell if the coordinates are permuted.

The following is the analytic cell decomposition theorem; see \cite{DM96} Section 4.2.
\begin{proposition}\label{cellDecomposition}
A definable set $D\subset\RR^n$ can be partitioned into a finite union of analytic cells.
\end{proposition}

We will use Proposition \ref{cellDecomposition} in several places. For example, we will partition the complement of a bad set $\{f= 0\}$ into pieces that are ``vertically connected,'' i.e.~each piece meets every vertical line in at most one interval. The precise formulation is as follows.

\begin{corollary}\label{partitionIntoVerticalConnectedSets}
Let $D\subset\RR^3$ be open and definable, and let $f\colon D\to\RR$ be analytic and definable. Then we can partition the set $\{z\in D \colon f(z) \neq 0\}$ into a definable set $Z\subset D$ of dimension at most $2$, plus a finite union of open connected sets, each of which is vertically connected, i.e.~they intersect each vertical line $\{(z_1,z_2)\}\times\RR$ in a connected set. 
\end{corollary}

\subsection{B\^ocher's theorem and its consequences}
The following theorem of B\^ocher characterizes when $n$ analytic functions are linearly dependent.
\begin{theorem}[B\^ocher \cite{Boc00}]\label{BocherThm}
Let $I\subset\RR$ be an interval and let $f_1,\ldots,f_n$ be analytic functions from $I$ to $\RR$. Then $f_1,\ldots,f_n$ are linearly dependent if and only if the $n\times n$ Wronskian determinant
\begin{equation}
W(t) = \left| 
\begin{bmatrix}
f_1(t) & \cdots & f_n(t)\\
f_1'(t) & \cdots & f_n'(t)\\
f_1^{(n-1)}(t) & \cdots &  f_n^{(n-1)}(t)
\end{bmatrix}
\right|
\end{equation}
vanishes identically on $I$. 
\end{theorem}
The next few results show how B\^ocher's theorem can help us understand the quantities $\mathfrak{r}_p^\gamma$ and $\mathfrak{r}_c^\gamma$ from Definition \ref{DefnTauAndB}. 

\begin{corollary}\label{corOfBocher}
Let $I\subset\RR$ be a closed interval, let $W\subset\RR^k$ be connected and open, and let $f_1,\ldots,f_n\colon I\times W\to\RR$ be analytic and definable. Let $m\leq n$. Then the set
\begin{align}\label{rankLeqM}
&\big\{ w\in W\colon \operatorname{rank}(f_1(\cdot,w),\ldots,f_n(\cdot,w))\leq m\big\}
\end{align}
is either all of $W$, or it is contained in a finite union of sets of the form $\{w\in W\colon G(w)=0\}$, where $G\colon W\to\RR$ is analytic, definable, and $\{w\in W\colon G(w)=0\}$ has dimension at most $k-1$.
\end{corollary}

\begin{corollary}\label{maxRankCor}
Let $I\subset\RR$ be a bounded closed interval. Let $W\subset\RR^2\times\RR^d$ be connected and open. Let $\gamma\colon I\times W\to\RR$ be analytic and definable. Then there are numbers $m_p$ and $m_c$, and an open, definable set
\[
D \subset \{(z,u)\in W\colon  \mathfrak{r}_p^\gamma(z,u)=m_p,\ \textrm{and}\ \mathfrak{r}_c^\gamma(z,u)=m_c\}
\]
whose complement $Z = W\backslash D$ has dimension at most $d+1$, and is a finite union of sets of the form $\{w\in W\colon G(w)=0\}$, where $G\colon W\to\RR$ is analytic and definable. 
\end{corollary}

B\^ocher's theorem characterizes when a set of analytic functions $f_1,\ldots,f_n\colon I\times W\to\RR$ are linearly dependent for each fixed $w\in W$. If this is the case, then for each fixed $w\in W$, some of these functions can be written as a linear combination of others. However, the coefficients of this linear dependence will depend on the choice of point $w\in W$. The next result says that locally, we can ensure that this dependence is analytic.

\begin{lemma}\label{analyticDependence}
Let $I\subset\RR$ be a closed interval, let $W\subset\RR^k$ be bounded and open, and let $f_1,\ldots,f_n\colon I\times W\to\RR$ be analytic. Let $q\in W$, and suppose that 
\begin{equation}\label{rankConditionForG}
\operatorname{rank}(f_1(\cdot,w),\ldots,f_m(\cdot,w))=\operatorname{rank}(f_1(\cdot,w),\ldots,f_n(\cdot,w))=m\quad\textrm{for all}\ w\ \textrm{in a neighbourhood of}\ q.
\end{equation}

Then there exists a neighbourhood $W_1\subset W$ of $q$ and analytic functions $C_{i,j}\colon W_1\to\RR$ so that for each $i=m+1,\ldots,n$, we have
\[
f_i(t,w) = \sum_{j=1}^m C_{i,j}(w)f_j(t,w)\quad\textrm{for all}\ (t,w)\in I\times W_1.
\] 
\end{lemma}
Though this result is standard, we could not locate a proof of this precise result in the literature, so we include it here for completeness. 
\begin{proof}
Since each of $f_{m+1},\ldots,f_n$ is handled in the same fashion, it suffices to consider the case $n=m+1$. For notational convenience, define $g=f_{m+1}$. Since $f_1(\cdot,w),\ldots,f_m(\cdot,w)$ are linearly independent for each $w\in W$, by Theorem \ref{BocherThm} we can select a point $t_0\in I$ so that
\[
A(w)=\det \left(
\begin{array}{c c c}
 f_1(t_0,w) & \cdots & f_m(t_0,w)\\
\partial_t f_1(t_0,w) & \cdots & \partial_t f_m(t_0,w)\\
 \vdots & \ddots & \vdots\\
 \partial_t^{m-1} f_1(t_0,w) & \cdots &  \partial_t^{m-1}f_m(t_0,w)\\
\end{array} \right)
\]
is non-zero at $w=q$, and hence is non-zero in a neighbourhood $W_1$ of $q$. 

Next, we claim that
\begin{equation}\label{mP1Determinant}
\det \left(
\begin{array}{c c c c}
 f_1(t_0,w) & \cdots & f_m(t_0,w) & g(t_0,w) \\
\partial_t f_1(t_0,w) & \cdots & \partial_t f_m(t_0,w) & \partial_t g(t_0,w)\\
 \vdots & \ddots & \vdots  & \vdots\\
 \partial_t^{m-1} f_1(t_0,w) & \cdots &  \partial_t^{m-1}f_m(t_0,w) & \partial_t^{m-1}g(t_0,w)\\
f_1(t,w)  & \cdots &  f_m(t,w) & g(t,w)
\end{array} \right)=0
\end{equation}
We verify this claim as follows. Fix $t$ and $w$. By \eqref{rankConditionForG}, there are numbers $a_1,\ldots,a_m$ so that
\begin{equation}\label{gInTermsOff}
g(s, w) = \sum_{i=1}^m a_i f_i(s,w)\ \qquad\textrm{for all}\ s\in I.
\end{equation} 
Differentiating \eqref{gInTermsOff}, we have
\begin{equation}\label{gInTermsOffDifferentiated}
\partial_s^k g(s, w) = \sum_{i=1}^m a_i \partial_s^k f_i(s,w)\ \qquad\textrm{for all}\ k=0,\ldots,m-1,\ s\in I.
\end{equation} 
We also have
\[
g(t, w) = \sum_{i=1}^m a_i f_i(t,w),
\]
and hence
\[
\left(\begin{array}{c}
g(t_0,w)\\
\partial_t g(t_0,w)\\
\vdots\\
\partial_t^{m-1}g(t_0,w)\\
g(t,w)
\end{array}
\right)
= \sum_{i=1}^m a_i \left(\begin{array}{c}
f_i(t_0,w)\\
\partial_t f_i(t_0,w)\\
\vdots\\
\partial_t^{m-1}f_i(t_0,w)\\
f_i(t,w)
\end{array}
\right),
\]
i.e.~the final column is a linear combination of the first $m$ columns. This establishes \eqref{mP1Determinant}.

Finally, let $B_i(w)$ be the determinant of the $m\times m$ minor of the matrix in \eqref{mP1Determinant} obtained by deleting the $i$-th column and final row. Observe that $B_i(w)$ is a sum of products of analytic functions, and hence is analytic. We also have that $B_{m+1}(w)\neq 0$ for $w\in W_1$. Taking a cofactor expansion of the determinant in \eqref{mP1Determinant}, we have
\[
B_{m+1}(w)g(t,w) = \sum_{i=1}^{m} (-1)^{m+i}B_i(w)f_i(t,w).
\]
The result now follows by setting $C_i(w)=(-1)^{m+i}B_i(w)/B_{m+1}(w).$
\end{proof}

\subsection{Proof of Proposition \ref{breakIntoConstantMpAndMc}}
We now have the tools to prove Proposition \ref{breakIntoConstantMpAndMc}. The next result decomposes $W$ into pieces on which $\mathfrak{r}_c$ and $\mathfrak{r}_p$ are constant.

\begin{lemma}\label{decomposeIntoCellsWithConstantMpMc}
Let $I\subset\RR$ be a bounded closed interval. Let $W\subset\RR^2\times\RR^d$ be open, connected, and definable. Let $\gamma\colon I\times W\to\RR$ be analytic and definable. Let $\Gamma$ be the family of plane curves associated to the defining function $\gamma$. 

Then there is a partition $W = \bigcup_{i=1}^N W_i\sqcup\bigcup_{i=1}^P X_i$, and pairs $(\tilde\gamma_i,\tilde W_i),$ $i=1,\ldots,N$, so that the following holds.
\begin{enumerate}[(i)]
    \item\label{WiOpenTildeGammaIDefinable} Each $\tilde W_i\subset\RR^2\times\RR^{d_i}$ is open, connected, and definable, and $\tilde\gamma_i\colon I\times\tilde W_i\to\RR$ is analytic and definable.
    \item Each $\gamma|_{W_i}$ and $\tilde\gamma_i$ determine the same family of curves. 
    \item\label{mpAndMcConstantOnCell} $\mathfrak{r}_c^{\tilde \gamma_i}$ and $\mathfrak{r}_p^{\tilde \gamma_i}$ are constant on each $\tilde W_i$.
    \item\label{formOfTildeGammaI} Each $\tilde\gamma_i$ is of the form 
    \[ 
    \tilde\gamma_i(t,z,v)=\gamma(t,z, U(z,v)),
    \] 
    where $(z,v)\mapsto(z,U(z,v))$ is an analytic, definable diffeomorphism between $\tilde W_i$ and $W_i$.  
    \item Each $X_i$ is definable, and $|\pi_z(X_i)|=0$.

\end{enumerate}
\end{lemma}
\begin{proof}
We will prove the result by (strong) induction on $d$. Both the base case and induction step will be similar, so to avoid repetition we will begin the same way. Apply Corollary \ref{maxRankCor} to $\gamma$ and $W$, and let $m_p,m_c$, $D$ and $Z = W\backslash D$ be the output of that Corollary. $\mathfrak{r}_p^\gamma(z,u)$ and $\mathfrak{r}_c^\gamma(z,u)$ are constant on $D$, both $Z$ and $D$ are definable, and $\dim Z \leq d+1$. 

Apply Proposition \ref{cellDecomposition} (analytic cell decomposition theorem) to $D$, and also to $Z$. This partitions $D$ and $Z$ into a finite union of analytic cells. Let $W_1,\ldots,W_M$ be the open cells in this decomposition (i.e.~the cells of type $(1,1,\ldots,1)$). Note that each of these cells is a subset of $D$, and hence $\mathfrak{r}_p^\gamma(z,u)$ and $\mathfrak{r}_c^\gamma(z,u)$ are constant on each $W_i$, $i=1,\ldots,M$. 

Let $X_1,\ldots,X_L$ be the cells with $|\pi_z(X_j)|=0$, and let $W_{M+1},\ldots,W_P$ be the remaining cells (i.e.~the cells that satisfy $|\pi_z(W_i)|>0$ and $\dim W_i\leq d+1$). Note that $\dim Z\leq d+1$ and thus when $d=0$, then every cell is either open, or satisfies $|\pi_z(X_j)|=0$, and in particular every cell in the partition of $Z$ is of this latter type.

Each cell $W_i$, $i=1,\ldots,P$ is of type $(1,1,j_1,\ldots,j_d)$ for some $j_1,\ldots,j_d$. Let $d_i=\sum j_\ell$. Observe that $\pi_z(W_i)$ is a $(1,1)$-cell. Define $\tilde W_i=\pi_z(W_i)\times(0,1)^{d_i}$. Let $G_i\colon \tilde W_i\to W_i$ be a definable analytic diffeomorphism of the form $(z,v)\mapsto  (z, U_i(z,v))$, with definable analytic inverse, and define $\tilde\gamma_i(t,z,v)=\gamma(t,z,U_i(z,v))$. Note that $\tilde\gamma_i$ is analytic and definable (since it is a composition of analytic and definable functions). We have that $\gamma|_{W_i}$ and $\tilde\gamma_i$ determine precisely the same family of curves.

If $W_i$ is an open cell (i.e. $1\leq i\leq M$), then $d_i=d$. Furthermore, $W_i\subset D$ and thus $\mathfrak{r}_c^{\gamma}$ and $\mathfrak{r}_p^{\gamma}$ are constant on $W_i$.  We claim that $\mathfrak{r}_c^{\tilde \gamma_i}$ and $\mathfrak{r}_p^{\tilde \gamma_i}$ are constant on $\tilde W_i$. To verify this, fix a point $(z,v)\in\tilde W_i$, and let $u = U_i(z,v)$. We use the chain rule to compute 
\[
\partial_{v_j}\tilde\gamma_i(t,z,v) = \sum_{k=1}^d \partial_{u_k}\gamma(t, z, U_i(z,v))\partial_{v_j}(U_i)_k(z,v),\quad j=1,\ldots,d,
\]
and 
\[
\partial_{z_j}\tilde\gamma_i(t,z,v) = \partial_{z_j}\gamma(t,z,U_i(z,v))+ \sum_{k=1}^d \partial_{u_k}\gamma(t, z, U_i(z,v))\partial_{z_j}(U_i)_k(z,v),\quad j=1,2.
\]

We can re-write these as the vector-valued equations about functions in $t$:
\[
\big(\partial_z\tilde\gamma_i(\cdot,z,v)\ \partial_v\tilde\gamma_i(\cdot,z,v)\big)=\big(\partial_z\gamma(\cdot,z,u)\ \partial_u\gamma(\cdot,z,u)\big)
\left(
\begin{array}{cc}
I_2 & 0 \\
D_zU_i(z,v) & D_vU_i(z,v)
\end{array}
\right),
\]
where $I_2$ denotes the $2\times 2$ identity matrix. Since $D_vU_i(z,v)$ is an invertible $d\times d$ matrix at each point $(z,v)\in \tilde W_i$, we have that for each fixed $(z,v)\in \tilde W_i$, 

\[
\operatorname{span}\{\partial_{v_1}\tilde\gamma_i(\cdot,z,v),\ldots,\partial_{v_d}\tilde\gamma_i(\cdot,z,v)\}=\operatorname{span}\{\partial_{u_1}\gamma(\cdot,z,u),\ldots,\partial_{u_d}\gamma(\cdot,z,u)\},
\]
and thus $\mathfrak{r}_p^{\tilde \gamma_i}(z, v)=\mathfrak{r}_p^\gamma(z, u)$, and the latter is constant for all $(z,u)\in W_i$. 

Similarly,
\begin{equation*}
\begin{split}
\operatorname{span} & \{\partial_{z_1}\tilde\gamma_i(\cdot,z,v),\partial_{z_2}\tilde\gamma_i(\cdot,z,v), \partial_{v_1}\tilde\gamma_i(\cdot,z,v),\ldots,\partial_{v_d}\tilde\gamma_i(\cdot,z,v)\}\\
& =\operatorname{span}\{\partial_{z_1}\gamma(\cdot,z,u),\partial_{z_2}\gamma(\cdot,z,u),\partial_{u_1}\gamma(\cdot,z,u),\ldots,\partial_{u_d}\gamma(\cdot,z,u)\},
\end{split}
\end{equation*}
and thus $\mathfrak{r}_c^{\tilde \gamma_i}(z, v)=\mathfrak{r}_c^\gamma(z, u)$, and the latter is constant for all $(z,u)\in W_i$. Thus Item \ref{mpAndMcConstantOnCell} (as well as Items \ref{WiOpenTildeGammaIDefinable} and \ref{formOfTildeGammaI}) is satisfied. If $d=0$ then every cell with $|\pi_z(W_i)|>0$ is open; this completes the proof of the base case $d=0$.

If $M+1\leq i\leq P$, then $|\pi_z(W_i)|>0$ and $d_i\leq d-1$, and hence $\tilde\gamma_i\colon I\times\tilde W_i\to\RR$ satisfies the induction hypothesis (i.e.~the hypotheses of Lemma \ref{decomposeIntoCellsWithConstantMpMc} with $d_i$ in place of $d$). We apply the induction hypothesis to $\tilde W_i$ and $\tilde\gamma_i$. This gives a decomposition $\tilde W_i=\bigsqcup_j \tilde W_{i,j}\sqcup \bigsqcup_{j} \tilde X_{i,j}$ and pairs $(\vardbtilde\gamma_{i,j}, \vardbtilde W_{i,j})$. Define $X_{i,j}=G_i(\tilde X_{i,j})\subset W_i$ and $W_{i,j}=G_i(\tilde W_{i,j})\subset W_i$. Thus $W_i=\bigsqcup_j W_{i,j}\sqcup \bigsqcup_j X_{i,j}$. Since $G_i$ preserves the $z$ variable, we have $\pi_z(X_{i,j})=\pi_z(\tilde X_{i,j})$, and hence $|\pi_z(X_{i,j})|=0$ for each index $j$.

Let $H_{i,j}(z,w)=(z,U_{i,j}(z,w))$ be the $z$-preserving analytic, definable diffeomorphism supplied by the induction hypothesis from $\vardbtilde W_{i,j}$ onto $\tilde W_{i,j}$. Then
\[
G_i\circ H_{i,j}(z,w)=\big(z,U_i(z,U_{i,j}(z,w))\big)
\]
is a $z$-preserving analytic, definable diffeomorphism from $\vardbtilde W_{i,j}$ onto $W_{i,j}$, with analytic, definable inverse. Moreover,
\[
\vardbtilde\gamma_{i,j}(t,z,w)=\gamma\big(t,z,U_i(z,U_{i,j}(z,w))\big).
\]
Thus Item \ref{formOfTildeGammaI} is satisfied. To conclude the proof, let $W_1,\ldots,W_N$, $X_{1},\ldots,X_P$, and $(\tilde\gamma_1,\tilde W_1),\ldots,(\tilde\gamma_N,\tilde W_N)$ be obtained by re-indexing the sets constructed in the steps described above. 
\end{proof}

Armed with this lemma, we can now prove Proposition \ref{breakIntoConstantMpAndMc}.

\begin{proof}[Proof of Proposition \ref{breakIntoConstantMpAndMc}]
By hypothesis, $W$ is open and definable, and $\gamma\colon I\times W\to\RR$ is analytic and definable. Apply Lemma \ref{decomposeIntoCellsWithConstantMpMc} to each connected component of $W$ (since $W$ is definable, there are finitely many such components). Abusing notation, let $W_1,\ldots,W_N$ be the sets from this decomposition for which $|\pi_z(W_i)|>0$ (and let $\tilde W_i,$ and $\tilde\gamma_i$ be the corresponding open sets and functions, as described in \ref{decomposeIntoCellsWithConstantMpMc}), and let $X_1,\ldots,X_P$ be the sets for which $|\pi_z(X_j)|=0$. Our pairs $(\tilde\gamma_i,\tilde W_i),$ $i=1,\ldots,N$ will be the output of the proposition. Conclusions \ref{NikodymIffOnGammaI}, \ref{constMcMpOnGammaI}, and \ref{massiveCurveReuseIffOnGammaI} are immediate. 

We will now establish Conclusion \ref{localReuseIffOnGammaI}. The reverse implication is straightforward: suppose there is an index $i$ so that $\tilde\gamma_i$ has local curve reuse not covered by foliations. Then the same is true for $\gamma|_{W_i}$, and hence also for $\gamma$. We will now establish the forward implication. Suppose that $\gamma$ (and hence each $\tilde\gamma_i$, $i=1,\ldots,N$) does not have massive curve reuse. Suppose in addition that each $\tilde\gamma_i$ does not have local curve reuse not covered by foliations. Thus for each index $i$, we must be in either Case \ref{CaseMc=Mp+2} or Case \ref{CaseMc=mp+1}.\ref{CaseMp=0} of Theorem \ref{characterizationOfNikodymSets}. In either case, there is a countable set of smooth vector fields $\mathcal{V}_i$ that strongly covers the locally reused curves of $\tilde\gamma_i$. Let $\mathcal{V}=\bigcup_{i=1}^N \mathcal{V}_i$. Note that Theorem \ref{characterizationOfNikodymSets} has not yet been proved (it will be proved in Section  \ref{ProofOfThmCharacterizationOfNikodymSetsSection}). However, our proof of Theorem \ref{characterizationOfNikodymSets} will not use Proposition \ref{breakIntoConstantMpAndMc}.

Let $B = \bigcup_{j=1}^P \overline{\pi_z(X_j)}$. Since $\pi_z(X_j)$ is definable and has dimension at most 1 for each $j=1,\ldots,P$, we have that $\overline{\pi_z(X_j)}$ has dimension at most 1, and hence $|B|=0$. Thus it suffices to prove that for each $(z,u)\in W$ with $z\not\in B$, either $\beta_{z,u}$ is not locally reused near $(z,u)$, or the curve $\beta_{z,u}$ is covered by $\mathcal{V}$. Suppose to the contrary that $\beta_{z,u}$ is locally reused near $(z,u)$, and is not covered by $\mathcal{V}$. Our goal is to establish a contradiction. Since $B$ is closed, we have $\dist(z,B)>0$, and hence for all sufficiently small neighbourhoods $W^\circ\subset W$ of $(z,u)$, we have that $\pi_z(W^\circ\cap\hat\beta_{z,u} \cap \bigcup_{i=1}^N W_i)$ contains a set of positive one-dimensional Hausdorff measure. Hence there exists an index $i$ so that $\pi_z(\hat\beta_{z,u} \cap W_i)$ contains a set of positive one-dimensional Hausdorff measure. Since $\pi_z(\hat\beta_{z,u}\cap W_i)$ is definable, we conclude that $\pi_z(\hat\beta_{z,u}\cap W_i)$ contains a curve (indeed, after applying Proposition \ref{cellDecomposition}, we may choose our curve to be a one-dimensional cell, or a line segment contained inside a two-dimensional cell). In particular, there is a point $(z',u')\in \hat\beta_{z,u}\cap W_i$ so that $\beta_{z,u}=\beta_{z',u'}$ is locally reused in a neighbourhood of $(z',u')$ within $W_i$. But this in turn means that $\beta_{z',u'}$ is locally reused by $\tilde\gamma_i$. Since $\mathcal{V}_i$ covers every curve that is locally reused by $\tilde\gamma_i$, we conclude that $\beta_{z,u}$ is covered by $\mathcal{V}_i$, and hence also by $\mathcal{V}$, which is a contradiction. 
\end{proof}


\section{Point rank, curve rank, and unused variables}\label{curveRankSection}
In this section we will consider the case where either the curve rank $\mathfrak{r}_c^\gamma$ or the point rank $\mathfrak{r}_p^\gamma$ is less than maximal. First, we will show that it is always possible to reduce to the case where $\mathfrak{r}_p^\gamma=d$. Second, we will show that if $\mathfrak{r}_p^\gamma=d$ and $\mathfrak{r}_c^\gamma<d+2$, then we have local curve reuse, and we can explicitly parameterize the reused curves.

\begin{lemma}\label{curveStraightening}
Let $I\subset\RR$ be a closed interval and let $W\subset\RR^a\times\RR^b\times\RR^c$ be open and non-empty. Let $\phi(t, v,w,\eta)\colon I\times W\to\RR$ be analytic. Suppose that for each $(v,w,\eta)\in W$, we have 
\begin{align}\label{rankCondition}
b & = \operatorname{rank}\{\partial_{w_1}\phi (\cdot, v,w,\eta),\ldots,\partial_{w_b}\phi(\cdot, v,w,\eta)\}\notag \\
& = \operatorname{rank}\{\partial_{w_1}\phi (\cdot, v,w,\eta),\ldots,\partial_{w_b}\phi(\cdot, v,w,\eta),\partial_{\eta_1}\phi(\cdot,v,w,\eta),\ldots,\partial_{\eta_c}\phi(\cdot,v,w,\eta) \},
\end{align}
i.e.~all $\partial_{w_i}\phi$ are linearly independent (as functions of $t$), while all $\partial_{\eta_j}\phi$ are linear functions of the $\partial_{w_i}\phi$.

Let $q=(v_0,w_0,\eta_0)\in W$. Then there is a neighbourhood $q\in W_1\subset W$ and analytic functions $\widetilde w_i(v,w,\eta)\colon W_1\to\RR$, $i=1,\ldots,b$ so that the following holds.
\begin{itemize}
\item[(i)] $\widetilde w(v,w,\eta_0)=(\widetilde w_1(v,w,\eta_0),\cdots,\widetilde w_b(v,w,\eta_0)) = w$.
\item[(ii)] The function $F(v,w,\eta) =(v,\widetilde w(v,w,\eta),\eta)$ is a diffeomorphism onto its image.
\item[(iii)] $\phi(t, v, \widetilde w(v,w,\eta), \eta)=\phi(t,v,w,\eta_0)$ for all $(t, v,w,\eta)\in I\times W_1$. 
\end{itemize}
\end{lemma}
\begin{proof}
When $b=0$ we have that $\phi$ is independent of $w$ and $\eta$, and the result is immediate. Suppose now that $b\geq 1$. By Lemma \ref{analyticDependence}, our rank condition \eqref{rankCondition} implies that there is a neighbourhood $W_1$ of $q$ and analytic functions $D_{i,j}\colon W_1\to\RR$ so that for each $i=1,\ldots,c$, we have
\begin{equation}\label{expressPartialZi}
    \partial_{\eta_i}\phi(\cdot, v,w,\eta) = \sum_{j=1}^b D_{i,j}(v,w,\eta)\partial_{w_j}\phi(\cdot, v,w,\eta).
\end{equation}
We will construct $\widetilde w$ as follows. For notational simplicity, we will suppose $(v_0,w_0,\eta_0)=(0,0,0)$. In what follows, we may abuse notation and replace $W_1$ by a smaller neighbourhood of $q$. For $0\leq s\leq 1$, define $Y(s,v,w,\eta)=(Y_1(s,v,w,\eta),\ldots,Y_b(s,v,w,\eta))\colon [0,1]\times W_1\to\RR^b$ to be the solution to the initial value problem
\begin{equation}\label{setUpODE}
\left\{ 
\begin{array}{l}
\frac{d}{ds}Y_j(s,v,w,\eta) = -\sum_{i=1}^c \eta_iD_{i,j}(v,Y(s,v,w,\eta),s\eta),\quad j=1,\ldots,b\\
\\
Y(0,v,w,\eta) = w
\end{array}
\right.
\end{equation}
$Y$ is analytic on $[0,1]\times W_1$, and hence if we define $\widetilde w(v,w,\eta) = Y(1,v,w,\eta)$, then $\widetilde w$ is analytic. If $\eta=0$ then $\frac{d}{ds}Y_j(s,v,w,0)=0$, and thus $Y(s,v,w,0)=w$ for all $s$, and hence $\widetilde w(v,w,0)=w$ (i.e. $\widetilde w(v,w,\eta_0) = w$). This establishes Item (i).

Next we can compute that $DF$ is a block matrix of the form
\[
DF(0,0,0) = \left(
\begin{array}{ccc}
I_a & 0 & 0 \\
* & I_b & *\\
0 & 0 & I_c
\end{array}
\right),
\]
where $I_a$ denotes the $a\times a$ identity matrix, and similarly for $I_b$ and $I_c$ (the blocks denoted by $*$ have potentially non-zero coefficients that we don't bother to calculate precisely). In particular, $\det DF(0,0,0)=1$, and hence $\det DF(v,w,\eta)\neq 0$ for $(v,w,\eta)$ in a neighbourhood of $q$. After further shrinking $W_1$ if necessary, we have that $F$ is a diffeomorphism onto its image. This establishes Item (ii).

Finally, we compute
\begin{align*}
    \frac{d}{ds} \phi(t, v, Y(s,v,w,\eta), s\eta) & = \sum_{j=1}^b \partial_{w_j}\phi(t, v, Y(s,v,w,\eta), s\eta)\frac{d}{ds}Y_j(s,v,w,\eta) + \sum_{i=1}^c \eta_i\partial_{\eta_i}\phi(t, v, Y(s,v,w,\eta), s\eta)\\
& = -\sum_{j=1}^b \partial_{w_j}\phi \sum_{i=1}^c \eta_i D_{i,j} + \sum_{i=1}^c \eta_i\partial_{\eta_i}\phi\\
& = \sum_{i=1}^c \eta_i\Big(\partial_{\eta_i}\phi - \sum_{j=1}^b D_{i,j}\partial_{w_j}\phi \Big)\\
& =0,
\end{align*}
where the second line used the definition of $Y(s,v,w,\eta)$ from \eqref{setUpODE}, and the final line used \eqref{expressPartialZi}. We conclude that $\phi(t, v, Y(s,v,w,\eta), s\eta)$ is independent of $s$, and in particular 
\[
\phi(t, v, w, \eta_0) = \phi(t, v, w, 0) = \phi(t, v, Y(0,v,w,\eta), 0)=\phi(t, v, Y(1,v,w,\eta), \eta)=\phi(t, v, \widetilde w(v,w,\eta),\eta).
\]
This establishes Item (iii).
\end{proof}

Applying Lemma \ref{curveStraightening} to $\gamma(t, z, u)$ with $a=2$ (i.e.~the $v$ variables in the lemma correspond to $z=(z_1,z_2)$), $b = m_p$ and $c = d-m_p$ (here we permute the indices $1,\ldots,d$ so that $\partial_{u_1}\gamma,\ldots,\partial_{u_{m_p}}\gamma$ are linearly independent), we have the following.
\begin{corollary}\label{reduceTo_d=m_p_Cor}
Let $I\subset\RR$ be a closed interval. Let $W\subset\RR^2\times\RR^d$ be open. Let $\gamma\colon I\times W\to\RR$ be analytic. Let $\Gamma$ be the family of plane curves associated to the defining function $\gamma$. Suppose that $\mathfrak{r}_p^\gamma=m_p$ and $\mathfrak{r}_c^\gamma=m_c$ are constant on $W$. Let $q\in W$. 

Then there is an open neighbourhood $W_1 \subset W$ containing $q$, an open set of the form $\tilde W=W_z \times \tilde W_u \subset \RR^2\times\RR^{m_p}$, and an analytic function $\tilde\gamma\colon I\times \tilde W\to\RR$ so that the following holds.
\begin{itemize}
\item $\mathfrak{r}_p^{\tilde \gamma}=m_p$ and $\mathfrak{r}_c^{\tilde \gamma}=m_c$ are constant on $\tilde W$.
\item The defining functions $\gamma|_{W_1}$ and $\tilde\gamma$ determine the same family of plane curves, i.e.~for each $z\in W_z$, we have
\[
\{\beta^\gamma_{z, u}\colon (z,u)\in W_1\}=\{\beta^{\tilde \gamma}_{z, u}\colon (z,u)\in \tilde W\}.
\]
\end{itemize}
As a consequence, if we define $\Gamma_1$ to be the family of curves associated to $\gamma|_{W_1}$, then $\Gamma_1=\tilde\Gamma$.
\end{corollary}

If $m_p=d$ and $m_c=d+1$, then we can apply Lemma \ref{curveStraightening} to $\gamma(t, z, u)$ with $a=0$, $b=d+1$ (here we use the variables $z_2,u_1,\ldots,u_d$) and $c=1$ (the variable $z_1$) to obtain the following.
\begin{corollary}\label{mc=mp+1_Lem}
Let $I\subset\RR$ be a closed interval. Let $W\subset\RR^2\times\RR^d$ be open. Let $\gamma\colon I\times W\to\RR$ be analytic. Let $\Gamma$ be the family of plane curves associated to the defining function $\gamma$. Suppose that $\mathfrak{r}_p^\gamma=d$ and $\mathfrak{r}_c^\gamma=d+1$ are constant on $W$. Let $q=(z_0,u_0)\in W$, where $z_0=(z_{1,0},z_{2,0})$. 

Then after permuting the roles of $z_1$ and $z_2$ if needed, there is an open neighbourhood $W_1\subset W$ containing $q$, and analytic functions $\tilde z_2(z,u)\colon W_1\to\RR$ and $\tilde u(z,u)\colon W_1\to\RR^d$ satisfying $\widetilde z_2(z_{1,0}, z_2, u)=z_2$ and $\tilde u(z_{1,0}, z_2, u) = u$ so that the function $F(z,u) = (z_1,\ \tilde z_2(z,u),\ \tilde u(z,u))$ is a local diffeomorphism, and for all $(z,u)\in W_1$, we have that 
\[
\gamma\big(t, z_1, \widetilde z_2(z,u),\ \tilde u(z,u)\big)=\gamma(t, z_{1,0},z_2,u).
\]
\end{corollary}


\section{Proof of Theorem \ref{characterizationOfNikodymSets}}\label{ProofOfThmCharacterizationOfNikodymSetsSection}
In this section we will prove Theorem \ref{characterizationOfNikodymSets}. Before doing so, it will be helpful to introduce the following definition. Recall Definition \ref{definitionOfNikodymSet}, which describes when a family of curves admits Nikodym sets. The following definition describes the ``largest possible'' $F$ for a given set $E$.

\begin{definition}\label{defnFGammaIW}
Let $I\subset\RR$ be an interval, $W \subset\RR^{2}\times\RR^d$, $\gamma\colon I\times W\to \RR$, and $E\subset\RR^2$. Define 
\begin{align}\label{260314e2_14}
    F^\gamma_{I,W}(E)=\left\{z: \text{There exists }u\text{ such that }(z;u)\in W\text{ and } |(\star)_1|> 0\right\},
    \end{align}
    where 
    \begin{equation}
        (\star)_1:= \big\{t\in I: (t,\gamma(t;z;u))\in E\big\}.
    \end{equation}
\end{definition}
In the above definition, $|(\star)_1|$ denotes the one-dimensional Lebesgue measure of $(\star)_1$ (a priori, we do not know if this set is measurable, but if $(\star)_1$ is not measurable then the condition $|(\star)_1|>0$ fails). Note that in the above definition, we make no assumptions about the regularity of the set $E$ (and hence have no guarantees about the set $F$). In practice, however, we will restrict attention to $E$ Borel, and $\gamma$ analytic.

\subsection{Case \ref{CaseMc=Mp+2}}
It suffices to prove that for every point $q_0\in W$, there is a neighbourhood $W_1\subset W$ so that the family of curves $\Gamma_1$ associated to $\gamma|_{W_1}$ does not admit Nikodym sets, and each set $\beta^*$ associated to $\Gamma_1$ is at most countable (this will imply that for each curve $\beta=\beta_{z,u}$ we have that $\beta^*$ is at most countable, and in particular does not contain a set with positive one-dimensional Hausdorff measure).

Fix a choice of $q_0\in W$. Without loss of generality we may suppose that $W$ is a sufficiently small neighbourhood of $q_0$. By Corollary \ref{reduceTo_d=m_p_Cor}, we may suppose that $m_p=d$. 

Let $E\subset\RR^2$ with $|E|=0$. Our goal is to prove that $|F^\gamma_{I, W}(E)|=0$.
Define the $(d+2)\times(d+2)$ matrix
\begin{align}
    \mathfrak{C}(t,z,u):=\begin{bmatrix}
        \partial_{z_1}\gamma & \partial_{z_2}\gamma & \partial_{u_1}\gamma & \cdots & \partial_{u_d}\gamma \\
        \partial_{z_1}\partial_t\gamma & \partial_{z_2}\partial_t\gamma & \partial_{u_1}\partial_t\gamma & \cdots & \partial_{u_d}\partial_t\gamma \\
        \cdots & \cdots  & \cdots & \cdots & \cdots \\
        \partial_{z_1}\partial_t^{d+1}\gamma & \partial_{z_2}\partial_t^{d+1}\gamma & \partial_{u_1}\partial_t^{d+1}\gamma & \cdots & \partial_{u_d}\partial_t^{d+1}\gamma
     \end{bmatrix}.
\end{align}
We have that $\mathfrak{C}(t,z,u)$ is analytic in $t,z,u$.  In particular, we can cover the set $\{ (t,z,u)\colon \det\big(\mathfrak{C}(t,z,u)\big)\neq 0\}$ by a countable union of sets of the form $I_j\times W_j$, with $I_j\subset I$ closed and $W_j\subset W$ open. Let $\gamma_j$ be the restriction of $\gamma$ to $I_j\times W_j$.  Applying Corollary \ref{corOfBddMaximalFn} to each such set, we conclude that
\[
|F^{\gamma_j}_{I_j,W_j}(E)|=0.
\]

Let $F_0$ be the set of points $z$ for which there exists $u$ so that $(z, u)\in W$, and
\[
|\{t\colon\det\big(\mathfrak{C}(t, z, u)\big)=0 \}|>0.
\]

Then
\[
F^\gamma_{I,W}(E) \subset F_0\cup \bigcup_j F^{\gamma_j}_{I_j,W_j}(E).
\]
However, since
\begin{equation}
    \mathfrak{r}_c^\gamma =d+2\ \textrm{on}\ W,
\end{equation}
by B\^ocher's theorem (Theorem \ref{BocherThm}) we see that for every $q=(z, u)\in W$, we have 
\begin{align}\label{260413e2_58c}
    \det\big(\mathfrak{C}(t,z,u)\big)\neq 0\quad\textrm{for all but finitely many}\ t\in I.
\end{align}
In particular, $F_0=\emptyset$. We conclude that $|F^\gamma_{I,W}(E)|=0$.

Finally, for each index $j$ and each pair $(z,u)\in W_j$, there are at most countably many $(z',u')\in W_j$ with $\beta_{z,u}=\beta_{z',u'}$. This is because for $t_0\in I_j$, the map $(z,u)\mapsto \big( \gamma(t_0,z,u),\partial_t\gamma(t_0,z,u),\ldots,\partial_t^{d+1}\gamma(t_0,z,u)\big)$ is a local diffeomorphism on $W_j$. We conclude that for each pair $(z,u)\in W$, there are at most countably many $(z',u')\in W$ with $\beta_{z,u}=\beta_{z',u'}$.

\medskip

\subsection{Case \ref{CaseMc=mp+1}.\ref{CaseMp=0}}
We claim it suffices to prove that for every point $q_0=(z_0,u_0)\in W$, there is a neighbourhood $W_1\subset W$ of $q_0$ so that the following holds. In what follows, $\Gamma_1$ is the family of curves associated to $\gamma|_{W_1}$. 
\begin{enumerate}[(i)]
\item\label{establishLocalCurveReuse} $\beta_{z_0,u_0}$ is reused near $(z_0,u_0)$
\item\label{Gamma1StronglyCovered} $\Gamma_1$ is strongly covered by foliations; indeed, \emph{every} curve $\beta_{z,u}$ with $(z,u)\in W_1$ is covered, not just the reused curves (though in Case \ref{CaseMc=mp+1}.\ref{CaseMp=0} the two are equivalent, since every curve is reused).
\item\label{noNikodymNoMassiveReuse} $\Gamma_1$ does not admit Nikodym sets, and no curve in $\Gamma_1$ is massively reused. 
\end{enumerate}
Item \ref{establishLocalCurveReuse}, proved for every $q_0\in W$, immediately implies its analogue in Theorem \ref{characterizationOfNikodymSets}, Case \ref{CaseMc=mp+1}.\ref{CaseMp=0}. Item \ref{noNikodymNoMassiveReuse} implies its analogue by Remark \ref{curveReuseIsLocalRemark}. To show that Item \ref{Gamma1StronglyCovered} implies its analogue, cover $W$ by a countable family of such open sets $\{W_i\}$ and take the union of the corresponding countable sets of vector fields. Since every curve in each local family is covered, the resulting countable family of vector fields covers every curve in $\Gamma$.

It remains to establish the above items. Fix a choice of $q_0=(z_0,u_0)\in W$, where $z_0=(z_{1,0},z_{2,0})$. By the discussion above (i.e.~abusing notation and replacing $W$ by $W_1$), we may suppose without loss of generality that $W$ is a sufficiently small neighbourhood of $q_0$. Applying Corollary \ref{reduceTo_d=m_p_Cor} followed by Corollary \ref{mc=mp+1_Lem}, we may suppose that $d=0$, and there is an analytic function $\widetilde z_2(z)$ satisfying $\widetilde z_2(z_{1,0}, z_2)=z_2$, so that $\gamma(t, z_1,\widetilde z_2(z))$ is independent of $z_1$, i.e.~$\gamma(t, z_1,\widetilde z_2(z))=h(t,z_2),$ where $\partial_{z_2} h(t,z_2)$ does not vanish identically. Each curve $\beta$ is of the form $\beta_{z_{1,0},z_2}=\{(t, h(t,z_2))\colon t\in I\}$, where $z_{1,0}$ is the first coordinate of the point $z_0$ chosen above, and $z_2$ ranges freely over a small neighbourhood of $z_{2,0}$. Note that for each $z_2$, we have that $\partial_{z_2} h(\cdot ,z_2)\not\equiv 0$, and hence $\partial_{z_2} h(t ,z_2) \neq 0$ for all but finitely many $t\in I.$

It is straightforward to verify that no curve is massively reused. Indeed, suppose that the curve $\beta_{z^\dag}$ is massively reused, i.e. the definable set $\{z^*\colon \beta_{z^*}=\beta_{z^\dag}\}$ has positive 2-dimensional Lebesgue measure. Since this set is definable, it must therefore contain a ball $B(z^{**},\eps)$. But this in turn forces $\partial_{z_1}\gamma(\cdot, z)=0$ and $\partial_{z_2}\gamma(\cdot, z)= 0$ for all $z$ in a neighbourhood of $z^{**}$, and in particular $\mathfrak{r}_c^\gamma(z^{**})=0$. By hypothesis, $\mathfrak{r}_c^\gamma(z)=1$ for all $z\in W$. 

To show that $\gamma$ is covered by foliations, we use the following lemma

\begin{lemma}
Let $I,J$ be intervals and let $h\colon I\times J\to\RR$ be analytic on a neighbourhood of $I\times J$, with $\partial_s h(\cdot,s)\not\equiv 0$ for each $s\in J$. Then there is a countable set of smooth vector fields $\{V_i\}$ so that for each $s\in J$ and a.e.~$t\in I$, there is an index $i$ so that $V_i(t, h(t,s))\neq 0$, and $V_i(t,h(t,s))$ is parallel to $(1, \partial_t h(t,s))$.
\end{lemma}
\begin{proof}
Fix a point $(t_0,s_0)$ where $\partial_s h(t_0,s_0)\neq 0$, and hence $G(t,s)=(t, h(t,s))$ is a local diffeomorphism in a neighbourhood of $(t_0,s_0)$. We can (locally) write the inverse as $G^{-1}(x,y) = (x, g(x,y))$. Now define the vector field $V(x,y) = (1, \partial_t h(x, g(x,y)))$, where $\partial_th$ denotes the derivative in the first variable. Then in a neighbourhood of $(t_0,s_0)$, we have  $V(t, h(t,s))=V(G(t,s))=(1, \partial_t h(t,s)).$ We can extend $V$ to a smooth vector field on $\RR^2$ by multiplying by a smooth bump function that is supported in a small neighbourhood of  $G(t_0, s_0)$.

Finally, select a countable set of points $(t_0,s_0)$ and corresponding neighbourhoods, that cover the set $\{(t,s)\colon \partial_s h(t,s)\neq 0\}$. This gives a countable set of vector fields $\{V_i\}$. For each $s\in J$, every point $t\in I$ with $\partial_s h(t,s)\neq 0$ is covered by one of these vector fields; the remaining set of points is at most countable, and in particular has measure 0.
\end{proof}

It remains to show that $\gamma$ does not admit Nikodym sets. By our initial discussion, it suffices to show that there is a neighbourhood $W_1$ of $z_0$ so that for every Borel set $E\subset\RR^2$ with $|E|=0$, we have
\begin{equation}\label{noMeasureOnE}
|F^\gamma_{I,W_1}(E)|=0.
\end{equation}

Fix such an $E$, and let $G(t,z_2) =  (t, h(t,z_2))$, and let $A = G^{-1}(E)=\{(t,z_2)\colon (t,h(t,z_2))\in E\}$. We have 
\[
DG(t, z_2) = \left( 
\begin{array}{cc}
1 & 0 \\
\partial_t h(t,z_2) & \partial_{z_2}h(t,z_2)
\end{array}
\right),
\]
which has determinant $\partial_{z_2}h(t,z_2)$. Thus by the area formula, we have 

\[
\int_A |\partial_{z_2}h(t,z_2)|dtdz_2=\int_E \#(A\cap G^{-1}(q))dq=0.
\]
Since $\partial_{z_2}h(t,z_2)\neq 0$ for a.e. pair $(t,z_2)$, it follows that $|A|=0$. Thus the set of points $z_2$ for which $|\{t\in I \colon (t, h(t,z_2))\in E\}|>0$ has measure 0. Denote this set $Z_2$. Finally, let $H(z_1,z_2) = (z_1, \tilde z_2(z_1,z_2))$, defined in a suitable neighbourhood of $z_0$. We have that $H(z_0)=z_0$, and $H$ is a local diffeomorphism. Select the neighbourhood $W_2$ sufficiently small so that $H|_{W_2}$ is a diffeomorphism. Then $F^\gamma_{I,H(W_2)}(E)\subset (H(\RR\times Z_2\cap W_2))$. Since $|\RR\times Z_2|=0$, we have $|H(\RR\times Z_2\cap W_2)|=0$ and hence $|F^\gamma_{I,H(W_2)}(E)|=0$. Since $H(z_0)=z_0$, we have found a neighbourhood $W_1=H(W_2)$ of $z_0$ satisfying \eqref{noMeasureOnE}, as desired.
\medskip

\subsection{Case \ref{CaseMc=mp+1}.\ref{CaseMp>0}}
It suffices to find a point $q_0\in W$ and a neighbourhood $W_1\subset W$ so that the family of curves $\Gamma_1$ associated to $\gamma|_{W_1}$ admits Nikodym sets, has curve reuse, and is not covered by foliations. We will choose $W_1$ so that $\bar W_1$ is compact and contained in $W$. In particular, $\gamma|_{W_1}$ is definable, as are various auxiliary functions defined using $\gamma$. 

\medskip

\noindent {\bf Step 1.}\\
Fix a point $q_0=(z_0,u_0)\in W$. For notational convenience we will suppose $(z_0,u_0)=(0,0)$. After shrinking $W$ (doing so will not change the conclusions of Case B.ii) and applying Corollary \ref{reduceTo_d=m_p_Cor}  followed by Corollary \ref{mc=mp+1_Lem}, we may suppose that $W = (-\eps, \eps)^2\times(-\eps,\eps)^d$; $m_p=d$; and there are definable analytic functions $\widetilde z_2(z,u)\colon W\to\RR$ and $\tilde u(z,u)\colon W\to\RR^d$ with $\widetilde z_2(0,z_2,u)=z_2$ and $\tilde u(0,z_2,u)=u$  so that 
\begin{equation}\label{gammaIndepZ1}
\gamma(t, z_1,\widetilde z_2(z,u),\tilde u(z,u))=\gamma(t, 0,z_2,u)\quad\textrm{for all}\ z_1\in(-\eps,\eps).
\end{equation}

Since $\partial_{z_2}\tilde z_2(0,z_2,u)=1$, after shrinking $W$ we may suppose that
\begin{equation}\label{z2NonVanishing}
\partial_{z_2}\tilde z_2(z,u)\neq 0\quad\textrm{for}\ (z,u)\in W.
\end{equation}

Note that the condition $m_p=d$ implies that $\partial_{u_1}\gamma(\cdot, z, u)$ is not identically 0 (as a function of $t$). Thus after further shrinking $W$ (and translating the origin if needed) and replacing $I$ by a possibly smaller interval, we may suppose that
\begin{equation}\label{u1NonVanishing}
\partial_{u_1}\gamma(t, z, u)\neq 0\quad\textrm{for}\ (t,z,u)\in I\times W.
\end{equation}
Note that we are free to replace $I$ by a smaller interval, since curve reuse, the existence of a Nikodym set, and the failure to be consistent with a foliation for this smaller interval will imply the same for the original interval.

Define ${\bf 0} = (0,\ldots,0)\in\RR^{d-1}$. Define 
\[
\phi^t(z_2,s) = \gamma(t, 0,z_2; s, {\bf 0}),\qquad \psi^{z_1}(z_2,s)=\widetilde z_2(z; s,{\bf 0}).
\] 
In the above, $\phi$ is a function of the variables $t,z_2,s$. We will often think of $\phi$ as a one-parameter family of maps $\phi^t\colon (z_2,s)\mapsto\phi^t(z_2,s)$. We will write $\nabla \phi$ to denote the gradient in the $(z_2,s)$ variables. Similarly for $\psi$. After a translation, the domains of these functions are $D_1=D_2=(-\eps,\eps)^3$. Note that $D_1$ has coordinates $(t,z_2,s)$, while $D_2$ has coordinates $(z_1,z_2,s)$. By \eqref{u1NonVanishing} we have $\partial_s \phi^t(z_2,s)\neq 0$ and hence $\nabla\phi^t(z_2,s)\neq 0$ on $D_1$.  By \eqref{z2NonVanishing} we have $\partial_{z_2}\psi^{z_1}(z_2,s)\neq 0$ and hence $\nabla\psi^{z_1}(z_2,s)\neq 0$ on $D_2$.

\medskip

\noindent {\bf Step 2.}\\
We claim that
\begin{equation}\label{nonVanishingGradPhi}
\partial_t \Big(\frac{\nabla\phi^t(z_2,s)}{\Vert \nabla\phi^t (z_2,s) \Vert}\Big)\quad\textrm{does not vanish identically on}\ D_1.
\end{equation}
We verify this claim as follows. Define $G(z,u)=(\widetilde z_2(z,u), \tilde u(z,u))$. Define $H_{z,s}(t) = (\partial_{z_2} \gamma,\partial_{u_1}\gamma,\ldots,\partial_{u_d}\gamma)^T$, evaluated at the point $(t, z_1, \widetilde z_2(z, s, {\bf 0}),\tilde u(z, s,{\bf 0}))$. The $d+1$ entries of $H_{z,s}(t)$ are linearly independent (as functions of $t$). We have
\[
\nabla\phi^t(z_2,s) = 
\left(\begin{array}{l}
 \partial_{z_2} G(z,s,{\bf 0})\cdot H_{z,s}(t)\\
 \partial_{u_1}G(z,s,{\bf 0})\cdot H_{z,s}(t)
 \end{array}\right).
\]
Since $F(z,u)=(z_1, \widetilde z_2(z,u),\tilde u(z,u))$ is a local diffeomorphism, we have that for each $(z,u)$, the vectors $\partial_{z_2} G(z,u)\in\RR^{d+1}$ and $\partial_{u_1}G(z,u)\in\RR^{d+1}$ are linearly independent. 

For $(z_2,s)$ fixed, we claim that the functions $A(t) = \partial_{z_2} \phi^t(z_2,s)$ and $B(t)=\partial_s \phi^t(z_2,s)$ are linearly independent as functions of $t$. Indeed, if this was not the case then there are numbers $a,b$ (not both zero) so that $aA(t)+bB(t)=0$ for all $t$, and thus 
\[
\big(a \partial_{z_2} G(z,s,{\bf 0}) + b \partial_{u_1}G(z,s,{\bf 0})\big)\cdot H_{z,s}(t)=0\quad\textrm{for all}\ t.
\]
Since the entries of $H_{z,s}(t)$ are linearly independent as functions of $t$, this would imply $a \partial_{z_2} G(z,s,{\bf 0}) + b \partial_{u_1}G(z,s,{\bf 0})=0$, but since $\partial_{z_2} G(z,s,{\bf 0})$ and $\partial_{u_1}G(z,s,{\bf 0})$ are linearly independent, this in turn implies that $a=b=0$. 

Finally, if \eqref{nonVanishingGradPhi} was false, then we could write $(A(t),B(t))=\lambda(t)(a,b)$ for some fixed vector $(a,b)$, but this contradicts the fact that $A(t)$ and $B(t)$ are linearly independent as functions of $t$. We conclude that \eqref{nonVanishingGradPhi} is true.

\medskip

\noindent {\bf Step 3.}\\
In this step we will show that the locally reused curves are not covered by foliations. For $(z_2,s)\in (-\eps,\eps)^2$, define the jet
\[
J_t(z_2,s) = \big( \gamma(t, 0,z_2; s, {\bf 0}),\ \partial_t \gamma(t, 0,z_2; s, {\bf 0}) \big) = \big(\phi^t(z_2,s),\ \partial_t \phi^t(z_2,s)\big).
\]
We can compute
\[
|\det DJ_t| = \Big|\det\left(
\begin{array}{cc}
\partial_{z_2} \phi^t & \partial_{z_2}\partial_t \phi^t \\
\partial_s    \phi^t & \partial_s \partial_t \phi^t
\end{array}
\right)
\Big| = \big\Vert \nabla \phi^t\big\Vert^2\ \Big\Vert\partial_t \Big(\frac{\nabla\phi^t}{\Vert \nabla\phi^t \Vert}\Big)\Big\Vert.
\]
We have that the RHS does not vanish identically on $D_1$ (we established this for $\Vert\nabla \phi^t\big\Vert^2$  in Step 1, and for $ \big\Vert\partial_t \big(\frac{\nabla\phi^t}{\Vert \nabla\phi^t \Vert}\big)\big\Vert$ in Step 2). Thus we can (temporarily) shrink the domain and find a smaller set $D_1'$ on which $J_t$ is a diffeomorphism, i.e.~the map
\[
(z_2,s)\mapsto \big( \gamma(t, 0,z_2; s, {\bf 0}),\ \partial_t  \gamma(t, 0,z_2; s, {\bf 0})\big)
\]
is a diffeomorphism. Shrinking $D_1'$ further, we will suppose it is a Cartesian product of three intervals $I_t\times I_{z_2}\times I_s$. If $V_i\colon\RR^2\to\RR^2$ is a smooth vector field, then the set of pairs $(z_2,s)\in I_{z_2}\times I_s$ with
\begin{equation}\label{tangentToVectorSpace}
V_i\big(t, \gamma(t, 0,z_2; s, {\bf 0})\big)\neq 0,\qquad V_i\big(t,\gamma(t, 0,z_2; s, {\bf 0})\big)\wedge\big(1, \partial_t  \gamma(t, 0,z_2; s, {\bf 0})\big)=0
\end{equation}
is a (possibly empty) one-dimensional smooth submanifold. Denote this set by $C_i(t)\subset I_{z_2}\times I_s$. Thus if $\{V_i\}$ is a countable set of smooth vector fields, then the set of pairs $(z_2,s)$ for which \eqref{tangentToVectorSpace} holds for some $i$ is contained in a countable union of smooth curves $\bigcup_i C_i(t)$, and thus has two-dimensional Lebesgue measure 0. We conclude that
\[
|\{(t, z_2, s)\in I_t\times I_{z_2}\times I_s \colon (z_2,s)\in \bigcup_i C_i(t)\}|=0.
\] 
If the curve $\beta_{(0, z_2), (s, {\bf 0})}$ is covered by $\{V_i\}$, then $(z_2,s)\in \bigcup_i C_i(t)$ for a.e.~$t\in I_t$. We conclude that
\[
|\{(z_2,s)\in I_{z_2}\times I_s\colon \beta_{(0, z_2), (s, {\bf 0})}\ \textrm{is covered by}\ \{V_i\} \}|=0,
\]
and hence we can select $s_0\in I_s$ so that
\[
|\{z_2\in I_{z_2}\colon \beta_{(0, z_2), (s_0, {\bf 0})}\ \textrm{is not covered by}\ \{V_i\} \}|=|I_{z_2}|>0.
\]
Denote the above set by $Z_2\subset I_{z_2}$. Recalling \eqref{gammaIndepZ1}, we have 
\begin{equation*}
\begin{split}
\{ (z_1,z_2) & \colon\textrm{exists}\ u\ \textrm{such that}\ \beta_{z_1,z_2,u}\ \textrm{is not covered by}\ \{V_i\} \}\\ 
&\supset \{ (z_1^*, \widetilde z_2(z_1^*, z_2^*; s_0,{\bf 0}))\colon z_1^*\in (-\eps,\eps),\ z_2^*\in Z_2\}.
\end{split}
\end{equation*}
The latter set has positive measure, since by \eqref{z2NonVanishing}, the map $(z_1^*,z_2^*)\mapsto (z_1^*, \widetilde z_2(z_1^*, z_2^*; s_0,{\bf 0}))$ is a local diffeomorphism. We conclude that the former set has positive measure. However, since every curve $\beta_{z_1,z_2,u}$ is locally reused near $(z_1,z_2,u)$, we conclude that the locally reused curves are not covered by $\{V_i\}$.

\medskip

\noindent {\bf Step 4.}\\
In this step, we consider the case where
\begin{equation}\label{vanishingGradPsi}
\partial_{z_1} \Big(\frac{\nabla\psi^{z_1}(z_2,s)}{\Vert \nabla\psi^{z_1} (z_2,s) \Vert}\Big)\quad\textrm{vanishes identically on}\ D_2.
\end{equation}
If \eqref{vanishingGradPsi} holds, then there is an analytic function $g$ so that $\widetilde z_2(z, s,{\bf 0})=\psi^{z_1}(z_2,s) = g(z_1, \psi^0(z_2,s))$. Since $\psi^0(z_2,s)=z_2$, this means $\widetilde z_2(z, s,{\bf 0})=g(z)$, i.e. $\widetilde z_2$ is independent of $s$. Abusing notation, we will write $\widetilde z_2(z)$ in place of $\widetilde z_2(z,s,{\bf 0})$. Since $\widetilde z_2(0,z_2)=z_2$, we have that $G(z)=(z_1,\widetilde z_2(z))$ is a diffeomorphism (and in particular, bi-Lipschitz) in a neighbourhood of $(0,0)$. Let $G^{-1}$ denote the inverse of this map, restricted to a suitably small neighbourhood of $(0,0)$. 

By \eqref{u1NonVanishing} we have $\partial_{u_1} \gamma(t, 0, z_2, u_1, {\bf 0})\neq 0$ on $I\times W$. Thus we may apply Corollary \ref{localWisewellThm} to conclude that there exists a null set $E\subset\RR^2$, and a positive-measure set $Y\subset(-\eps,\eps)$ so that for each $z_2\in Y$, there exists $s(z_2)$ so that the curve
\[
\alpha_{z_2, s(z_2)} = \{ (t, \gamma(t, 0, z_2, s(z_2), {\bf 0}))\colon t\in I\}
\]
satisfies $\mathcal{H}^1(\alpha_{z_2, s(z_2)}\cap E)>0$. 

Since $\beta_{(z_1, \widetilde z_2(z)), \tilde u(z, s(z_2), {\bf 0})}=\alpha_{z_2, s(z_2)}$ for $z_1\in(-\eps,\eps)$, we conclude that for all $(z_1, \widetilde z_2)\in G\big( (-\eps,\eps)\times Y\big)$ (this set has positive measure), we have that 
\[
\mathcal{H}^1(E \cap \beta_{(z_1, \widetilde z_2(z)), \tilde u(z, s(z_2), {\bf 0})}) > 0.
\]
Thus $\gamma$ admits Nikodym sets.

\noindent {\bf Step 5.}\\
Finally, we consider the case where \eqref{vanishingGradPsi} fails, i.e.~
\begin{equation}\label{nonVanishingGradPsi}
\partial_{z_1} \frac{\nabla\psi^{z_1}(z_2,s)}{\Vert \nabla\psi^{z_1} (z_2,s) \Vert}\quad\textrm{does not vanish identically on}\ D_2.
\end{equation}

To summarize, we have functions $\phi^t(z_2,s)$ and $\psi^{z_1}(z_2,s)$ satisfying \eqref{nonVanishingGradPhi} and \eqref{nonVanishingGradPsi}, respectively. By Proposition \ref{prescribedProjectionsTwoDirections}, there exists a Borel measurable set  $G\subset(-\eps,\eps)^2$ with the following properties:
\begin{enumerate}[(a)]
    \item\label{psiG>0} There is a positive measure set of $z_1\in(-\eps,\eps)$ for which $|\psi^{z_1}(G)|>0$. 
    \item\label{phiG=0} There is an interval $I'\subset (-\eps,\eps)$ so that for a.e.~$t\in I'$, we have $|\phi^t(G)|=0$. 
\end{enumerate}

Define
\[
    E_0 = \bigcup_{t\in I'}(t, \phi^t(G))=(I'\times\RR)\cap \bigcup_{(z_2, s)\in G}\beta_{(0,z_2),(s,{\bf 0})}
\]
$E_0$ is measurable, because it is the projection of the \emph{Borel} set $\{(t,y,z_2,s)\colon t\in I',\ (z_2,s)\in G,\ y = \gamma(t, 0, z_2, s, {\bf 0})\}$ to the first two coordinates. Since $E_0$ is measurable, Item \ref{phiG=0} and Fubini's theorem imply that $|E_0|=0$. Thus we can select a Borel set $E\supset E_0$ with $|E|=0$. 

As a consequence of Item \ref{psiG>0}, we have that our \emph{original} region $W$ satisfies $|F^\gamma_{I,W}(E)|>0$. Hence $\Gamma$ admits Nikodym sets.

\subsection{Case \ref{CaseMp=mc}}
After permuting the indices $1,\ldots,d$, we may suppose that $\partial_{u_1}\gamma,\ldots,\partial_{u_{m_p}}\gamma$ are independent. Let $(z_0,u_0)\in W$ and apply Lemma \ref{curveStraightening} to $\gamma(t, z, u)$ with $a=d-m_p$, $b = m_p$, $c=2$ (i.e.~the $v$ variables correspond to $u_{m_p+1},\ldots,u_d$, the $w$ variables correspond to $u_1,\ldots,u_{m_p}$, and the $\eta$ variables correspond to $z=(z_1,z_2)$). 

We obtain an open neighbourhood $W_1\subset W$ of $(z_0,u_0)$, and an analytic function $\tilde u\colon W_1\to\RR^d$ so that $\gamma(t, z,\tilde u(z,u))=\gamma(t, z_0,u)$ for all $(z,u)\in W_1$. After further shrinking $W_1$, we may suppose that $W_1$ is a Cartesian product $W_1 = W_z\times W_u$.  The curve $\beta=\beta_{z_0,u_0}$ satisfies $\beta^*\supset \pi_z(W_1)$. Since $W_1$ is open, we have $|\pi_z(W_1)|>0$, and hence $\beta$ is massively reused. In particular, by setting $E = \beta$ we have $|E|=0$ and $|F^\gamma_{I,W}(E)|\geq|\beta^*|>0$, so $\gamma$ admits Nikodym sets.


\section{Proof of Theorem \ref{NikodymIffCurveReuse}}\label{ProofOfThmNikodymIffCurveReuseSection}
We conclude by showing how Theorem \ref{characterizationOfNikodymSets} and Proposition \ref{breakIntoConstantMpAndMc} combine to prove Theorem \ref{NikodymIffCurveReuse}. Let $I\subset\RR,$ $W\subset\RR^2\times\RR^d$, and $\gamma\colon I\times W\to\RR$ be as in the statement of Theorem \ref{NikodymIffCurveReuse}. First, if $W$ is a (bounded) ball and $\gamma$ extends to an analytic function on a neighbourhood of $I\times\bar W$, then the conclusion of Theorem \ref{NikodymIffCurveReuse} follows by applying Proposition \ref{breakIntoConstantMpAndMc} followed by Theorem \ref{characterizationOfNikodymSets}.

For the general case, cover $W$ by a countable collection of bounded balls $\{W_i\}$, where $\overline W_i\subset W$ for each index $i$. Let $\Gamma_i$ be the family of curves associated to $\gamma|_{W_i}$. By Remark \ref{curveReuseIsLocalRemark}, we have that 
\begin{itemize}
\item $\Gamma$ has massive curve reuse if and only if the same is true for at least one $\Gamma_i$.
\item $\Gamma$ admits Nikodym sets if and only if the same is true for at least one $\Gamma_i$. 
\end{itemize}
Suppose that $\Gamma$ does not have massive curve reuse (and hence neither does any $\Gamma_i$). To conclude the proof, we must show that $\Gamma$ has local curve reuse not covered by foliations if and only if the same holds for at least one $\Gamma_i$. It is straightforward to see that if at least one $\Gamma_i$ has local curve reuse not covered by foliations, then the same is true for $\Gamma$. 

We now argue in the opposite direction. Suppose that each $\Gamma_i$ does not have local curve reuse not covered by foliations, i.e.~for each index $i$, there is a countable set of vector fields $\mathcal{V}_i$ so that $|Q^{\gamma_i}(W_i,\mathcal{V}_i)|=0$ (recall the definition of this set from \eqref{setOfLocallyReusedCuvesCoveredByVi}). Let $\mathcal{V}=\bigcup \mathcal{V}_i$. If $(z,u)\in W$ and $\beta_{z,u}$ has local curve reuse near $(z,u)$, then there exists an index $i$ such that $\beta_{z,u}$ has local curve reuse in $W_i$ near $(z,u)$. Thus $Q^{\gamma}(W,\mathcal{V}) \subset \bigcup_i Q^{\gamma_i}(W_i,\mathcal{V}_i)$, and hence $|Q^{\gamma}(W,\mathcal{V})|=0$, i.e.~$\Gamma$ does not have local curve reuse not covered by foliations.


\appendix


\section{The axis-parallel elliptic maximal operator}\label{ellipticMaximalOperatorSection}
In this section we will show that the axis-parallel elliptic maximal operator is bounded in $L^p$ for all $p>3$. We begin with a precise definition. For $f\colon\RR^2\to\CC$ and $z=(z_1,z_2)\in\RR^2$, define the axis-parallel elliptic maximal operator
\[
\mathcal{E}f(z) = \sup_{a,b\in[1,2]}\int_0^{2\pi} |f(z_1+a\cos\theta, z_2+b\sin\theta)|d\theta.
\]
This maximal operator was first introduced by Lee, Lee, and Oh in \cite{LLO25}. In that paper it is called the ``strong circular maximal function.'' Lee, Lee and Oh proved that $\mathcal{E}$ is bounded in $L^p(\RR^2)$ for $p>4$, and they proved that it is unbounded for $p\leq 3$.

\begin{theorem}\label{axisParallelElliptic}
The axis-parallel elliptic maximal operator $\mathcal{E}$ is bounded on $L^p$ for all $p>3$.
\end{theorem}

To prove Theorem \ref{axisParallelElliptic}, we first denote by $c<1/10$ a small enough absolute constant (independent of $p$) to be determined at the end of Section \ref{sub:proof_L3}. By a trivial partition and rescaling, it will be helpful to restrict the ranges of $a$ and $b$ to a smaller interval, say, $[1,1+c]$. Also, by symmetry, it suffices to prove the bound for $\mathcal E$ with the interval of integration changed to $\theta\in [\frac \pi 4,\frac{3\pi}4]$, which then in particular implies that $a |\cos \theta|<0.8$. This allows us to write the ellipses as graphs of functions over $[-0.8,0.8]$, and it suffices to consider the following maximal function instead:
\begin{equation}
    \mc{M}_{\ellipse}f(z):= 
    \sup_{a, b\in [1,1+c]}
    \left|\int_{\R} f\Big(z_1-t, 
    z_2-b\sqrt{
    1-\pnorm{
    \frac{t}{a}
    }^2
    }
    \Big)  \chi(t; a, b)dt\right|,
\end{equation}
where $\chi(t; a, b): \R\times \R^2\to \R$ is a smooth bump function that vanishes outside $(-0.9,\,0.9)\times (0.9,\,1.2)^2$. By a simple tiling argument, Theorem \ref{axisParallelElliptic} is thus a consequence of the following estimate. 
\begin{proposition}\label{230609theorem2_1}
    We have the maximal operator bound for every $p>3$:
    \begin{equation}
        \norm{
        \mc{M}_{\ellipse} f
        }_{L^p([0,c]^2)}
        \lesim_{p, \chi}
        \norm{f}_{L^p(\R^2)}.
    \end{equation}
\end{proposition}
The remainder of this section is devoted to proving Proposition \ref{230609theorem2_1}. (For simplicity, we will drop the dependence on $\chi$ for the rest of this proof. Also, we will abbreviate $\sup_{a,b\in [1,1+c]}$ as $\sup_{a,b}$.)

\subsection{Interpolation}

Let $\delta>0$ be a small number. Define a $\delta$-thickened version of $\mc{M}_{\ellipse}$ by
\begin{equation}
    \mc{M}_{\ellipse, \delta}f(z):= 
    \sup_{a, b}
    \frac{1}{\delta}
    \left|\int_0^{\delta}\int_{\R} f\Big(z_1-t, 
    z_2-b\sqrt{
    1-\pnorm{
    \frac{t}{a}
    }^2
    }-y'
    \Big)  \chi(t; a, b)dt dy'\right|.
\end{equation}
To prove Proposition \ref{230609theorem2_1}, first, by interpolating with the trivial bound at $p=\infty$, it suffices to prove this for all $p\in (3,12]$. To this end, it suffices to prove the following estimates for arbitrarily small $\varepsilon>0$:
\begin{align}
    &\norm{
    \mc{M}_{\ellipse, \delta}f
    }_{L^3([0,c]^2)}
    \lesim_{\varepsilon} \delta^{-\varepsilon}
    \norm{f}_{L^3(\R^2)},\label{230609e2_6}\\
    &\norm{
    \mc{M}_{\ellipse}{P_k f}
    }_{L^{12}([0,c]^2)}
    \lesim 2^{-\beta k}
    \norm{f}_{L^{12}(\R^2)},\label{260918_TY}
\end{align}
for some fixed $\beta>0$, where $P_k f$ denotes a Littlewood-Paley frequency projection onto the annulus of diameter $2^k$. Indeed, taking $\delta=2^{-k}$, \eqref{230609e2_6} implies that
\begin{equation}\label{230609e2_6_TY}
    \norm{
    \mc{M}_{\ellipse}{P_kf}
    }_{L^3([0,c]^2)}
    \lesim_{\varepsilon} 2^{k\varepsilon}
    \norm{f}_{L^3(\R^2)},
\end{equation}
so by interpolating \eqref{260918_TY} with \eqref{230609e2_6_TY}, we obtain for all $p\in (3,12]$ that
\begin{equation}\label{20260918_TY2}
    \norm{
    \mc{M}_{\ellipse}{P_kf}
    }_{L^p([0,c]^2)}
    \lesim_{\varepsilon}2^{-k(\frac{4\beta}{3}(1-\frac{3}{p})-\varepsilon)}
    \norm{f}_{L^p(\R^2)}.
\end{equation}
Since $p>3$, choosing $\varepsilon>0$ small enough (depending on $p,\beta$) and summing over $k\ge 1$ gives Proposition \ref{230609theorem2_1}.

\subsection{Proof of $L^3$ endpoint estimate}\label{sub:proof_L3}

We now prove \eqref{230609e2_6}. Such bounds will essentially follow from \cite[Theorem 1.7]{PYZ22}, which we invoke as the following Lemma.

\begin{lemma}\label{lemma:PYZ_TY}
    Let $\delta>0$ be a small number. Let $J\subset\mathbb R$ be a fixed compact interval, and let $\chi(t;a,b)$ and $\phi(t;v,a,b)$ be compactly supported smooth functions such that
\begin{equation}\label{230609e2_17}
\det
\begin{bmatrix}
\partial_v\phi & \ \partial_a\phi & \ \partial_b \phi\\
\partial_v\partial_{t}\phi & \ \partial_a\partial_{t}\phi &  \ \partial_b\partial_{t} \phi\\
\partial_v\partial_{t}^2\phi & \ \partial_a\partial_{t}^2\phi & \ \partial_b\partial_{t}^2 \phi
\end{bmatrix}
\neq 0
\end{equation}
at every point $(t;v,a,b)$ such that $v\in J$ and
$(t;a,b)\in\operatorname{supp}\chi$.
For $v\in J$, define the maximal operator 
\begin{equation}
    \mathfrak{M}_{\delta}f(v):=\sup_{a, b}
    \frac{1}{\delta}
    \anorm{
    \int_{\R}\int_0^{\delta}
    f(t, \phi(t; v, a, b)-y')  \chi(t; a, b)dt dy'
    }.
\end{equation}
Then for all $\varepsilon,\delta\in (0, 1)$, 
\begin{equation}
\|\mathfrak M_\delta f\|_{L^3(J)}
\lesssim_{\varepsilon,\phi,\chi,J}
\delta^{-\varepsilon}\|f\|_{L^3(\mathbb R^2)}.
\end{equation}

\end{lemma}
To apply Lemma \ref{lemma:PYZ_TY}, we break the operator $\mc{M}_{\ellipse, \delta}$ into two parts as follows
\begin{equation}\label{230609e2_10}
    \begin{split}
        & 
        \sup_{a, b}
    \frac{1}{\delta}
        \left|\int_0^{\delta}\int_{-c}^{c} f\Big(z_1-t, 
    z_2-b\sqrt{
    1-\pnorm{
    \frac{t}{a}
    }^2
    }-y'
    \Big)  \chi(t; a, b)dt dy'\right|\\
    & + 
    \sup_{a, b}
    \frac{1}{\delta}
    \left|\int_0^{\delta}\int_{|t|\ge c} f\Big(z_1-t, 
    z_2-b\sqrt{
    1-\pnorm{
    \frac{t}{a}
    }^2
    }-y'
    \Big)  \chi(t; a, b)dt dy'\right|.
    \end{split}
\end{equation}
When applying Lemma \ref{lemma:PYZ_TY} below, we take absolute values and
use fixed nonnegative smooth majorants of the truncated amplitudes, supported where $|a-1|,|b-1|<2c$. For the latter term in \eqref{230609e2_10}, choose the majorant supported
where $c/2<|t|<0.95$. For the former term, after the substitution $x=z_1-t$,
choose a majorant independent of $z_1\in[0,2c]$ and supported where $|x|<4c$. The determinant computations below ensure nondegeneracy on these supports when $c$ is sufficiently small.

To control the contribution from the latter term in \eqref{230609e2_10}, we freeze the $z_1$-variable and denote 
\begin{equation}
\phi(t; z_2, a, b):=z_2-b\sqrt{
    1-\pnorm{
    \frac{t}{a}
    }^2
    }.
\end{equation}
We directly check the curvature condition given in \eqref{230609e2_17} for the function $\phi(t;z_2, a, b)$, and see that the determinant is equal to 
\begin{equation}
\frac{2 b t^3}{a (a^2 - t^2)^3}.
\end{equation}
This computation explains the decomposition in \eqref{230609e2_10}. It therefore remains to control the contribution from the former term in \eqref{230609e2_10}. By the change of variables 
\begin{equation}
z_1\mapsto z_1, \ \ z_2\mapsto z_1+z_2,
\end{equation}
it suffices to prove that 
\begin{equation}
\Norm{
\sup_{a, b}
    \frac{1}{\delta}
    \Big|\int_0^{\delta}\int_{-c}^{c} f\Big(z_1-t, 
    z_1-b\sqrt{
    1-\pnorm{
    \frac{t}{a}
    }^2
    }-y'
    \Big)  \chi(t; a, b)dt dy'\Big|
}_{
L_{z_1}^3([0,2c])
}
\lesim_{\varepsilon} \delta^{-\varepsilon} \norm{f}_{L^3(\R^2)}.
\end{equation}
To apply Lemma \ref{lemma:PYZ_TY}, we consider 
\begin{equation}
\tilde \phi(t; z_1, a, b):=z_1-
b\sqrt{
1-\pnorm{
\frac{z_1-t}{a}
}^2
}.
\end{equation} 
We check \eqref{230609e2_17} at $z_1=t=0$, $a=b=1$ for $\tilde \phi$, which equals $-2$. Thus, by continuity, if the absolute constant $c$ is chosen to be small enough, then \eqref{230609e2_17} has absolute value at least $1$. This finishes the estimate for the first term in \eqref{230609e2_10}, and thus the $L^3$ endpoint estimate \eqref{230609e2_6}.

\subsection{Proof of local smoothing estimate}
We now prove \eqref{260918_TY}. Let $\chi_{2^{-k\eta}}: \R\to \R$ be an $L^{\infty}$ normalized smooth bump function adapted to the interval $(-2^{-k\eta}, 2^{-k\eta})$. We split $\mc{M}_{\ellipse}$ into two parts by 
\begin{equation}
\begin{split}
&  \sup_{a, b}
    \Big|\int_{\R} f\Big(z_1-t, 
    z_2-b\sqrt{
    1-\pnorm{
    \frac{t}{a}
    }^2
    }
    \Big) \chi_{2^{-k\eta}}(t) \chi(t; a, b)dt\Big|\\
    &+ 
     \sup_{a, b}
    \Big|\int_{\R} f\Big(z_1-t, 
    z_2-b\sqrt{
    1-\pnorm{
    \frac{t}{a}
    }^2
    }
    \Big) (1-\chi_{2^{-k\eta}}(t)) \chi(t; a, b)dt\Big|.
\end{split}
\end{equation}
Let us write them as $\mc{M}'_{\ellipse}$ and $\mc{M}''_{\ellipse}$ separately. For the former term, we have the trivial bound 
\begin{equation}
\norm{
\mc{M}'_{\ellipse} f
}_{L^{\infty}([0,c]^2)}\lesim 2^{-k\eta}\norm{f}_{L^{\infty}(\R^2)}. 
\end{equation}
The kernel bound for $P_k$ and \eqref{230609e2_6},
applied to positive averages, give
\begin{equation}
\norm{\mc{M}'_{\ellipse}P_kf}_{L^3([0,c]^2)}
\lesim_\eta 2^{k\eta}\norm{f}_{L^3(\R^2)}.
\end{equation}
Interpolating with the $L^\infty$ bound gives
\begin{equation}
\norm{\mc{M}'_{\ellipse}P_kf}_{L^{12}([0,c]^2)}
\lesim_\eta 2^{-k\eta/2}\norm{f}_{L^{12}(\R^2)}.
\end{equation}
To bound the latter term, we will apply Theorem \ref{averagingOp}. Via a direct computation, the determinant in \eqref{soggeLocalMultiParam} is equal to 
\begin{equation}
\frac{
6  b^2 t 
}{
(a^2 - t^2)^{9/2}
}.
\end{equation}
In absolute values, this is $\gtrsim 2^{- k\eta}$. We therefore apply Theorem \ref{averagingOp} (see Remark \ref{remark:260920_TY} below) and obtain 
\begin{equation}
\norm{
\mc{M}''_{\ellipse} P_k f
}_{L^{12}([0,c]^2)}\lesim 2^{
-\frac{k}{12}
+C  k\eta
}\norm{f}_{L^{12}(\R^2)},
\end{equation}
where $C$ is a large universal constant. Take $\eta:=1/(24C)>0$ and set
$\beta=\min\{\eta/2,\,1/12-C\eta\}>0$.
Combining the two bounds proves \eqref{260918_TY}.

\begin{remark}\label{remark:260920_TY}
    We actually need a slightly stronger form of Theorem \ref{averagingOp}. More precisely, suppose that the absolute value of the determinant in \eqref{soggeLocalMultiParam} is uniformly bounded below by some $\kappa>0$. Then the local smoothing estimate \eqref{260920_TY} can be strengthened to
    \begin{equation}
        \Vert A^\gamma_\chi P_k f\Vert_{L^p(\RR^2\times\RR^d)}\lesssim \kappa^{-C_d} 2^{-(\frac{d+1}{p}-\eps) k}\Vert f\Vert_{L^p(\RR^2)},
    \end{equation}
    where $C_d$ depends only on $d$, and the implicit constant can be taken to depend only on the (say) $C^{d+10}$ norm of $\gamma$ as well as $p,d,\varepsilon,\chi$. Indeed, this can be deduced from the proof of Theorem \ref{averagingOp} in Section \ref{maximalFnSection}; in particular, in the decoupling Lemma \ref{d-decoupling}, the implicit constant in \eqref{eqn:decoupling_TY} can be chosen to depend polynomially on the reciprocal of the infimum of the volume of the parallelepiped formed by the vectors in $D_w (\mathcal T(\Psi))$. See, for instance, the proof of \cite[Proposition 5.21]{LY24}, which can be adapted to the case of Theorem \ref{averagingOp}.
\end{remark}


\section{Example \ref{noNikodymUnboundedMaximal}}\label{exampleNoNikodymUnboundedMaximalSection}
We discuss Example \ref{noNikodymUnboundedMaximal} more carefully in this section. First, we will show that the associated maximal operator $M^\gamma$ is unbounded on $L^p(\RR^2)$ for all $p<\infty.$ We will verify this in two steps. The first step is to create a sequence of maximal operators $\{N_s\}_{s>0}$, and show that for $0<s\leq 1$ we have
\begin{equation}
\lim_{s\searrow 0} \Vert N_s\Vert_{L^p(\RR^2)\to L^p(\RR^2)}=\infty.
\end{equation}
The second step is to show that 
 \begin{equation}
 \Vert M^\gamma\Vert_{L^p(\RR^2)\to L^p(\RR^2)} \geq \Vert N_s\Vert_{L^p(\RR^2)\to L^p(\RR^2)}\quad\textrm{for each}\ s>0.
 \end{equation}

We begin with the first step. For $s>0$ and $z\in (0,1)^2$, define 
\[
N_s f(z)=\sup_{u\in[0,1]}\inf_{s_1\in(0,s]}\int_0^1 |f(t, z_2+u(t-z_1)+s_1u^2(t-z_1)^2)|dt.
\]
Define $N_sf(z)=0$ for $z\notin (0,1)^2$. We claim that for $p<\infty$ fixed, 
\begin{equation}
\lim_{s\searrow 0}\Vert N_s\Vert_{L^p(\RR^2)\to L^p(\RR^2)}=\infty.
\end{equation}
To verify this claim, fix $\eps>0$. Let $E\subset[0,1]^2$ be a classical Nikodym set, i.e.~for each $z\in[0,1]^2\backslash E$, there is a line $\beta_{z,u}$ passing through $z$ with slope $u$ so that 
 \begin{equation}
 E\cap \beta_{z,u} = [0,1]^2\cap\beta_{z,u}\backslash \{z\}.
 \end{equation}
  After applying a reflection and 90$^\circ$ rotation if needed, we can find a Borel set $F\subset[0,1]^2$ with $|F|>1/10$ so that for each $z\in F$, the corresponding line $\beta_{z,u}$ satisfies $\mathcal{H}^1(\beta_{z,u}\cap E)\geq 1/10$, and $u\in[0,1]$. Let $E^\circ$ be a bounded open set containing $E$, with $|E^\circ|\leq\eps$. For $\delta>0$, define 
  \begin{equation}
  E_\delta=\{z\in E^\circ\colon B(z,\delta)\subset E^\circ\}.
  \end{equation}
   Since $E^\circ$ is open, for each $z\in E^\circ$ there exists $\delta_0>0$ so that $z\in E_\delta$ for all $\delta\leq\delta_0$. Thus $E_\delta\nearrow E^\circ$ as $\delta\searrow 0$. In particular, if we select $\delta>0$ sufficiently small, then $|E_\delta|\geq |E^\circ|/2$ and after replacing $F$ by a slightly smaller set (say with $|F|\geq 1/20)$, we have $\mathcal{H}^1(\beta_{z,u}\cap E_\delta)\geq 1/20$ for each $z\in F$. Choose $f\in C_c^\infty(E^\circ)$ with $0\leq f\leq1$
and $f=1$ on $E_{\delta/2}$. Then
$\Vert f\Vert_{L^p(\RR^2)}\leq\eps^{1/p}$.
If $0<s\leq\delta/2$, then
$|s_1u^2(t-z_1)^2|\leq\delta/2$
for all $s_1\in(0,s]$ and $u,t,z_1\in[0,1]$.
Thus this perturbation sends points of $E_\delta$
into $E_{\delta/2}$, where $f=1$.
We conclude that if $0<s\leq\delta/2$ then
\[
N_sf(z) \geq \inf_{s_1\in (0, s]}\int_0^1 |f(t, z_2+u(t-z_1)+s_1u^2(t-z_1)^2)|dt\geq \frac{1}{40}\quad\textrm{for all}\ z\in F,
\]
and hence
\begin{equation}
\Vert N_sf\Vert_{L^p(\RR^2)}\geq \frac{1}{800} \geq \frac{1}{1600}\eps^{-1/p}\Vert f\Vert_{L^p(\RR^2)}.
\end{equation}
 This concludes Step 1.

For the second step, let $f\colon\RR^2\to[0,\infty)$ be smooth,
let $0<s\leq1$, and define $f_s(z_1,z_2)=f(z_1,z_2/s)$.
We have
 \[
 \Vert f_s\Vert_{L^p(\RR^2)}=s^{1/p}\Vert f\Vert_{L^p(\RR^2)},\quad\textrm{and}\quad M^\gamma f_s(z_1, sz_2)\geq N_s f(z_1,z_2)\quad\textrm{for}\ z\in(0,1)^2,
 \] 
 and thus
 \begin{equation}
 \frac{\Vert M^\gamma f_s\Vert_{L^p(\RR^2)}}{\Vert f_s\Vert_{L^p(\RR^2)}} \geq s^{1/p} \frac{\Vert N_s f\Vert_{L^p(\RR^2)}}{\Vert f_s\Vert_{L^p(\RR^2)}}\geq \frac{\Vert N_s f\Vert_{L^p(\RR^2)}}{\Vert f\Vert_{L^p(\RR^2)}}.
 \end{equation}
We conclude that $M^\gamma$ is unbounded on $L^p$.

Next, we will show that the family of curves $\Gamma$ associated to $\gamma$ from Example \ref{noNikodymUnboundedMaximal} does not admit Nikodym sets. For each $u_0\in[0,1]$, let $\gamma_{u_0}$ be the restriction of $\gamma$ to the domain $[0,1]\times[0,1]^2\times[u_0,1]$. Let $M^{\gamma_{u_0}}$ be the corresponding maximal function. It is straightforward to verify that for $0<u_0<1$, $\gamma_{u_0}$ satisfies the cinematic curvature condition \eqref{soggeLocal} on its domain (indeed, the determinant is $8u^3\geq 8u_0^3$), and hence  $M^{\gamma_{u_0}}$ is bounded on $L^p(\RR^2)$ for all $p>2$  (though the operator bound diverges to $\infty$ as $p\searrow 2$ and as $u_0\searrow 0$). Now let $E\subset\RR^2$ be Borel with $|E|=0$, and let $F$ be the set of points for which there exists $u\in[0,1]$ with $\mathcal{H}^1(\beta_{z,u}\cap E)>0$. Our goal is to show that $|F|=0$. 

Let $F'\subset F$ consist of those $z\in F$ for which we may take $u=0$, and for $n\geq 1$, let $F_{n}'$ consist of those $z\in F$ for which we may take $u\geq 1/n$. We have $F = F'\cup \bigcup_{n=1}^\infty F_{n}'.$ It is straightforward to verify that $|F'|=0$---this is precisely the ``boring'' situation described in Example \ref{dumbNikodymNonExample}. On the other hand, since the family of curves associated to the defining function $\gamma_{1/n}$ does not admit Nikodym sets, we must have that $|F'_n|=0$ for all $n$.


\normalem


\begin{thebibliography}{}



\bibitem[BHS20]{BHS20} Beltran D., Hickman J. and Sogge C. \emph{Variable coefficient Wolff-type inequalities and sharp local smoothing estimates for wave equations on manifolds.} Analysis $\&$ PDE, 2020, 13(2): 403--433.





\bibitem[Boc00]{Boc00} B\^ocher, M. \emph{The theory of linear dependence.} Ann. of Math. (2) 2 (1900/1901) 81–96.




\bibitem[BCR13]{BCR13} Bochnak, J., Coste, M. and Roy, M. \emph{Real Algebraic Geometry.}  Vol. 36. Springer Science \& Business Media, 2013.



\bibitem[BT23]{BT23} R.~Bongers, K.~Taylor. \emph{Transversal families of nonlinear projections and generalizations of Favard length.} Anal. PDE 16:279--308, 2023.


\bibitem[Bou86]{Bou86} Bourgain, J. \emph{Averages in the plane over convex curves and maximal operators.} J. Analyse Math. 47 (1986), 69--85.


\bibitem[BDG16]{BDG16} Bourgain, J., Demeter, C. and Guth, L. \emph{Proof of the main conjecture in Vinogradov's mean value theorem for degrees higher than three.} Annals of Mathematics (2016): 633--682.



\bibitem[CC19]{CC19} A.~Chang, M. Cs\"ornyei. \emph{The Kakeya needle problem and the existence of Besicovitch and Nikodym sets for rectifiable sets.} Proc. London Math. Soc. 118:1084--1114, 2019.


\bibitem[CFMMT26]{CFMMT} A. Chang, R. Fraser, A. McDonald, M. Meyer, and K. Taylor. \emph{Prescribed projections for a general class of maps and efficient covering by variable plane curves}.  arXiv:2609.31602v1 (2026+).


\bibitem[CMT23]{CMT23} A.~Chang, A.~McDonald, K.~Taylor. \emph{Prescribed projections and efficient coverings by curves in the plane.} To appear, Anal. PDE.    arXiv:2310.08776. 



\bibitem[CG24]{CG24} Chen, M. and Guo, S. \emph{The dichotomy of Nikodym sets and local smoothing estimates for wave equations.} arXiv preprint arXiv:2402.15476.

\bibitem[CGY23]{CGY23} Chen, M., Guo, S., and Yang, T. \emph{A multi-parameter cinematic curvature}. arXiv preprint   arXiv:2306.01606. 


\bibitem[Dav52]{Dav52} R.~O.~Davies. \emph{On accessibility of plane sets and differentiation of functions of two real variables.} Proc. Cambridge Philos. 
Soc., 48:215--232, 1952.


\bibitem[Dries09]{Dries09} L.~van den Dries. \emph{Tame Topology and O-minimal Structures.} Cambridge University Press, 2009.


\bibitem[DMM94]{DMM94} L.~van den Dries, A. Macintyre, D. Marker. The elementary theory of restricted analytic fields with exponentiation. \emph{Annals of Math.} 140 (1994), 183–205.


\bibitem[DM96]{DM96} L.~van den Dries and C.~Miller. Geometric categories and o-minimal structures. \emph{Duke Math. J.} 84 (1996), 497–540.



\bibitem[Erd03]{Erd03} \erdogan, B. \emph{Mapping properties of the elliptic maximal function.} Rev. Mat. Iberoamericana 19 (2003), 221–234.




\bibitem[Fal85]{Fal85} Falconer, K. \emph{The geometry of fractal sets.} Cambridge Tracts in Mathematics 85 (Cambridge University Press, 1985).
%

\bibitem[Fal86]{Fal86} K.~Falconer. \emph{Sets with prescribed projections and Nikodym sets.} Proc. London Math. Soc. 53(1):48--64, 1986.









\bibitem[LLO25]{LLO25} Lee, J., Lee, S. and Oh, S. \emph{The elliptic maximal function.} Journal of Functional Analysis 288, no. 1 (2025): 110693.

\bibitem[Lep76]{Lep76} Lepson, B. \emph{On a problem of Peter Fenton and the distance set of the Cantor set.} Notices Amer. Math. Soc. 23 (1976) A-507.



\bibitem[LY24]{LY24} Li, J. and Yang, T. \emph{Two principles of decoupling}. arXiv:2407.16108v3. (2024+).


\bibitem[MR98]{MR98} Marletta, G. and Ricci, F. \emph{Two-parameter maximal functions associated with homogeneous surfaces in  $\R^n$.} Studia Mathematica 130.1 (1998): 53--65.



\bibitem[Mar87]{Mar87} Marstrand, J. \emph{Packing circles  in the plane.} Proc. London Math. Soc. 55(1987), 37-58.






\bibitem[MSS92]{MSS92} Mockenhaupt, G., Seeger, A. and Sogge, C. \emph{Wave front sets, local smoothing and Bourgain's circular maximal theorem.} 
Ann. of Math. (2) 136 (1992), no. 1, 207--218.




\bibitem[MSS93]{MSS93} Mockenhaupt, G., Seeger, A. and Sogge, C. \emph{Local smoothing of Fourier integral operators and Carleson-Sjölin estimates.} Journal of the American Mathematical Society 6, no. 1 (1993): 65--130.


\bibitem[Nik27]{Nik27} Nikodym, O. \emph{Sur la mesure des ensembles plans dont tous les points sont rectilineairement accessibles.} Fund. Math., 10 (1927) 116–168.



\bibitem[PS07]{PS07} Pramanik, M. and Seeger, A. \emph{$L^p$ regularity of averages over curves and bounds for associated maximal operators. } Amer. J. Math. 129.1 (2007): 61--103.


\bibitem[PYZ22]{PYZ22} Pramanik, M., Yang, T. and Zahl, J. \emph{A Furstenberg-type problem for circles, and a Kaufman-type restricted projection theorem in $\R^3$.} arXiv:2207.02259, to appear in Amer. J. Math. 

\bibitem[Sch98]{Sch98} Schlag, W. \emph{A geometric proof of the circular maximal theorem.} Duke Math. J. 93 (1998),
no. 3, 505--533.



\bibitem[Sog91]{Sog91} Sogge, C. \emph{Propagation of singularities and maximal functions in the plane.} Inventiones mathematicae, volume 104, pages 349--376 (1991).



\bibitem[Wis04]{Wis04} Wisewell, L. Families of surfaces lying in a null set. \emph{Mathematika} 51 (2004), 155--162.

\bibitem[Zah26]{Zah26} Zahl J. \emph{On Maximal Functions Associated to Families of Curves in the Plane.} Duke Math. J. 175.2 (2026): 199--286.



\end{thebibliography}
\end{document}